\documentclass[11pt,twoside]{amsart}
 
\usepackage{amsfonts}
\usepackage{epsfig,graphics}
\usepackage{amssymb}
\usepackage{amscd}
\usepackage{hyperref}
\hypersetup{
  linktocpage,
    colorlinks,
    citecolor=black,
    filecolor=black,
    linkcolor=blue,
    urlcolor=blue
    }

\newtheorem{theoreme}{Theorem}[section]

\newtheorem{proposition}[theoreme]{Proposition}
       
\newtheorem{definition}[theoreme]{Definition}

\newtheorem{lemme}[theoreme]{Lemma}

\newtheorem{remarque}[theoreme]{Remark}

\newtheorem{corollaire}[theoreme]{Corollary}

\newcommand{\Heis}{{\operatorname{Heis}}}

\newcommand{\OO}{\operatorname{O}}

\newcommand{\PO}{\operatorname{PO}}

\newcommand{\Ein}{\operatorname{Ein}}

\newcommand{\Conf}{\operatorname{Conf}}

\newcommand{\Ad}{\operatorname{Ad}}
\newcommand{\Aut}{\operatorname{Aut}}

\newcommand{\calh}{{\mathcal{H}}}

\newcommand{\tm}{\tilde{{M}}}

\newcommand{\BZ}{{\bf Z}}
\newcommand{\NN}{{\bf N}} 
\newcommand{\RR}{{\bf R}}

\newcommand{\hm}{{\widehat{M}}}
\newcommand{\hx}{{\hat{x}}}
\newcommand{\hy}{{\hat{y}}}

\newcommand{\delx}{{\frac{\partial}{\partial x}}}
\newcommand{\dely}{{\frac{\partial}{\partial y}}}
\newcommand{\delz}{{\frac{\partial}{\partial z}}}

\newcommand{\heis}{{\operatorname{\mathfrak{heis}}}}

\newcommand{\liez}{{\mathfrak{z}}}
\newcommand{\lieu}{{\mathfrak{u}}}
\newcommand{\liezx}{{\mathfrak{z}_{\scriptscriptstyle{X}}}}
\newcommand{\lieg}{{\mathfrak{g}}}
\newcommand{\lieh}{{\mathfrak{h}}}
\newcommand{\liep}{{\mathfrak{p}}}

\newcommand{\oo}{{\mathfrak{o}}}

\newcommand{\liea}{{\mathfrak{a}}}

\newcommand{\lieq}{{\mathfrak{q}}}

\newenvironment{preuve}{\medskip \noindent {\bf Proof: }}
   {$\diamondsuit$ }

\begin{document}

\title[3-dimensional Lorentzian Lichnerowicz Conjecture]{The Lorentzian Lichnerowicz Conjecture for real-analytic,
  three-dimensional manifolds}
\author{Charles Frances and Karin Melnick}
\thanks{The second author was partially supported by a Joan and Joseph
Birman Fellowship for Women Scientists and NSF Award DMS-2109347}
\maketitle
\tableofcontents

\linespread{1.1}
\section{Introduction}

The Lichnerowicz Conjecture in conformal Riemannian geometry was proved
simultaneously by J. Ferrand and M. Obata.  Recall that the conformal
transformations of a semi-Riemannian manifold $(M,g)$ form the group
$$ \mbox{Conf}(M,[g]) = \{ f \in \mbox{Diff}(M) \ : \ f^*g = e^{2
  \lambda} g, \ \lambda \in C^\infty(M) \}$$
and that this is a Lie group provided $\mbox{dim }M \geq 3$.
A subgroup $H \leq \mbox{Conf}(M,[g])$ is called \emph{essential} if
it does not act isometrically with respect any
$g' = e^{2 \lambda} g$ in the conformal class $[g]$ of $g$.  The
identity component of a Lie group $H$ is denoted by $H^0$.

\begin{theoreme}[Ferrand '71 \cite{lf.lich}/ Obata '71 \cite{obata.lich}]
 \label{thm.riem.lich}
 Let $(M,g)$ be a  compact, Riemannian manifold with dimension $n \geq
 2$. If $\mbox{Conf}^{\;0}(M,[g])$ is essential, then 
  $(M,g)$ is conformally diffeomorphic to the round sphere ${\bf S}^n$.
\end{theoreme}

The first attempt to characterize the round sphere 
 by this property of its conformal group seems to have been by A. Lichnerowicz \cite{lich.cr} 
around 1964. 
Ferrand actually proved a stronger result for all $n \geq 2$, that the
above conclusion
holds whenever $\mbox{Conf}(M,[g])$ is essential. (In dimension 2, the theorem is a
straightforward consequence of the 
uniformization theorem for Riemann surfaces.)  She later proved a statement for noncompact $M$ in
\cite{lf.noncompact}. The reader will find in \cite{ferrand.histoire}
a nice account by her of the subject.

Obata's proof is based on techniques from
differential geometry and transformation groups, while Ferrand's is
based on quasiconformal analysis.  Two more, totally different proofs,
also covering the noncompact case, 
were given in 1995 by Schoen \cite{schoen.cr}, based on geometric
PDEs, in particular, scalar curvature theory, and in 2007 by the first
author \cite{frances.lfrank1}, using Cartan connections and
dynamical techniques.

The question whether there is a higher-signature analogue of theorem
\ref{thm.riem.lich} has been around for about thirty years (see
\cite[Sec 6.7]{dag.rgs}).  Note that essentiality
of a conformal action on a compact Riemannian manifold is equivalent to noncompactness of
the group.  In higher-signature, there
is a wide, largely uncharted array of compact pseudo-Riemannian
manifolds with noncompact isometry group.  Accordingly, there is not
such a simple characterization of which conformal groups can act
essentially.

Even compact Lorentzian manifolds with essential conformal group occur
in a wide variety of global geometries.  The first author found infinitely-many topological types of compact manifolds, for each
$n \geq 3$, supporting infinitely-many nonequivalent Lorentzian conformal structures admitting an
essential conformal flow, in \cite{frances.nogloballich}.  Locally,
however, they are all conformally equivalent to Minkowski space---that
is, all known essential Lorentzian examples are \emph{conformally
  flat}.  The conjecture is:

\begin{quotation}{\bf Lorentzian Lichnerowicz Conjecture (LLC)}.  \emph{Let
  $(M^n,g)$ be a compact Lorentzian manifold with $n \geq 3$.  If
  $\mbox{Conf}(M,[g])$ is essential, then $(M,g)$ is conformally flat.}
  \end{quotation}

For pseudo-Riemannian metrics of type $(p,q)$ with $p,q \geq 2$, there
are rather simple, polynomial deformations $g$ of the flat,
$(p,q)$-Minksowski metric such that a compact quotient of
$(\RR^n,g) \backslash \{ 0 \}$, diffeomorphic to $S^1 \times
S^{n-1} $ for $n = p+q$, is not conformally flat and admits an essential flow \cite{frances.pq.counterexs}.  Thus it seems there is no
version of the Lichnerowicz Conjecture true in signature
higher than Lorentzian.

In this article, we prove the Lichnerowicz Conjecture for
3-dimensional, real-analytic Lorentzian manifolds:
  
\begin{theoreme}
 \label{thm.main.theorem}
 Let $(M,g)$ be a $3$-dimensional, compact, real analytic, Lorentzian manifold. If $\mbox{Conf}^{\;0}(M,[g])$ is essential, then 
  $(M,g)$ is conformally flat.
\end{theoreme}

  \subsection{Previous work on the conjecture}

  It is nearly understood which connected Lie groups can act conformally and essentially on
  compact Lorentzian manifolds.  The expectation
  is that any such group admits a local monomorphism into
  $\mbox{O}(2,n)$.

  Let $(M,g)$ be a compact, pseudo-Riemannian manifold
of type $(p,q)$, $p+q \geq 3$, and assume $p \leq q$.  Let $H \leq \mbox{Conf}(M,[g])$ be a
connected subgroup.

For $H$ semisimple, Zimmer proved that $\mbox{rk}_{\RR} H \leq p+1$
\cite{zimmer.rankbounds}; moreover, if this rank is attained, then $H$
necessarily acts essentially.
Bader--Nevo proved that if $H$ is simple and attains the maximal $\RR$-rank, then it is
locally isomorphic to $\mbox{O}(p+1,k+1)$, for $p \leq k \leq q$
\cite{bn.simpleconf}.  Under the same assumptions, the first author
and Zeghib subsequently proved that
$M$ is
conformally flat, and in fact conformally equivalent to
a certain compact, conformally homogeneous model space, up
to covering spaces \cite{fz.simpleconf}.  For $g$ Lorentzian,
Pecastaing has shown that if $H$ is noncompact, simple, and essential,
then $(M,g)$ is conformally flat \cite{pecastaing.smooth.sl2r}.

For $H$ nilpotent, the authors proved in \cite{fm.nilpconf} that the
nilpotence degree of $H$ is at most $2p+1$, and that, when this
maximal degree is attained, $(M,g)$ is conformally flat and again
equivalent to the homogeneous model, up to covering spaces.  Moreover,
if $H$ has the maximal nilpotence degree, it necessarily acts essentially.

A recent result of the second author and Pecastaing \cite{mp.confdambra}, supporting the LLC,
does not assume any structure on the group, as above, but rather topological properties of the space.  The theorem states that the conformal group of a
compact, simply connected, analytic Lorentzian manifold is compact.
The proof shows that noncompactness of $H$ implies conformal
flatness.  By D'Ambra's Theorem \cite{dambra.lorisom}, $H$ noncompact
is equivalent to $H$ essential for such spaces.  Conformal
flatness leads to a contradiction of the simple connectedness
assumption.  The proof reduces to the case that the group is abelian.

\subsection{Compact three-dimensional Lorentzian manifolds}

One of our motivations for theorem \ref{thm.main.theorem} was the 
thorough understanding of isometries of compact, 3-dimensional
Lorentzian manifolds.  In \cite{zeghib.lorentz.3d}, Zeghib classified
all such spaces admitting an unbounded isometric flow.  The first
author recently improved this classification to all such spaces
admitting any noncompact isometry group---including in particular the
case where the isometry group is infinite and discrete \cite{frances.lorentz.3d}.

There are moreover many useful classifications of 
homogeneous models for 3-dimensional Lorentzian manifolds (eg,
\cite{calvaruso.kowalski.ricci.oper,coley.hervik.pelavas.3d,sekigawa.curv.homog.3d,rahmani.lorentz.heis}).
In \cite{dz.lor3d.lochom}, Dumitrescu and Zeghib classified all
metrically homogeneous Lorentzian spaces ${\bf X}$ such that there is
a compact, 3-dimensional Lorentzian manifold locally modeled on ${\bf
  X}$, and they proved that these are all complete.

\subsection{Overview of proof}

The proof rests on the approach we have developed to conformal
Lorentzian transformations in our previous papers, which in turn is
based on techniques involving the Cartan connection associated
to a conformal structure and on Gromov's results on automorphisms of rigid geometric
structures.  We moreover draw
on some of the work specific to 3-dimensional Lorentzian manifolds referenced
above as well as the recent advance in \cite{mp.confdambra}, from
which we draw two major parts of our proof.

In section \ref{sec.existence}, we use Zeghib's classification of
unbounded 3-dimensional Lorentzian flows \cite{zeghib.lorentz.3d} to
show that $(M,g)$ has an essential conformal vector field.
Denoting such a vector field by $X$, we ultimately focus on the Lie algebra
$\liezx$ of local conformal vector fields commuting with $X$.

In section \ref{sec.curvature.vanishing} we gather local results yielding conformal flatness,
based on our previous work and Gromov's theory, applied in this
3-dimensional, analytic context.  These are used throughout the
paper, and they immediately imply that the dimension of $\liezx$ is at most 4.

The remainder of the paper comprises four more or less distinct
proofs, for each of the cases, $\mbox{dim } \liezx$ equals 4, 3, 2, or 1.
The case $\mbox{dim } \liezx =4$ corresponds to $(M,[g])$ being
locally conformally homogenous, which quickly leads to the conclusion
that it is conformally flat, or $X$ is inessential, a contradiction.

When $\mbox{dim } \liezx = 3$, it can be $\RR^3$,
$\mathfrak{heis}(3)$, or $\mathfrak{aff}(\RR) \oplus \RR$.
In the case of $\mathfrak{heis}(3)$, we explicitly find a coordinate
chart exhibiting $g$ as conformally flat.  For $\mathfrak{aff}(\RR)
\oplus \RR$, we find a complete $(G,{\bf X})$-structure on a closed,
invariant surface, and use this to show that the flow along $X$ on this surface
gives rise to conformal flatness.

When $\mbox{dim } \liezx = 2$, it is isomorphic to $\RR^2$.  We show
in section \ref{sec.global-action} that $\liezx$ globalizes and integrates to a
cylinder action on $M$.  Then the situation strongly resembles that of
\cite[Sec 6]{mp.confdambra}; in the remainder of section \ref{sec.2dimensional}, we follow the
outline of that proof to reach the desired conclusion.

Finally, when $\mbox{dim }\liezx =1$, we use fixed points of the
flow along $X$,
guaranteed by Gromov's theory, to alternately reach a contradiction or
conclude conformal flatness.  The proof in this case follows section 5
of \cite{mp.confdambra}.

\section{Existence of an essential vector field}
\label{sec.existence}

The objective of this section is to prove that, under the hypothesis that
$\mbox{Conf}^{0}(M,[g])$ is essential, there exists an
{\it essential conformal vector field}, namely, a vector field 
generating an essential conformal flow.  We denote by
$\mathcal{X}^{conf}(M)$ the space of all conformal vector fields on $M$.  Our proof will be specific to
3 dimensions, but will not require analyticity.

\begin{proposition}
  \label{prop.existence.essential.flow}
Let $(M,g)$ be a compact, smooth, 3-dimensional Lorentzian manifold.  If
$\mbox{Conf}^{\,0}(M,[g])$ is essential, then it contains an essential  1-parameter subgroup.
 In fact, every 1-parameter subgroup which is not relatively compact
 is essential. 
  \end{proposition}

\begin{preuve}
If $\mbox{Conf}^{\,0}(M,[g])$ is essential, then it is necessarily
noncompact.
Let
$\{ \varphi^t_X \} < \mbox{Conf}^{\,0}(M,[g])$ be an unbounded
1-parameter group.  Suppose, for a contradiction, that 
$\{ \varphi^t_X \}$ is inessential, that is, contained in
$\mbox{Isom}^0(M,g')$ for some $g' \in [g]$.  In this case, Zeghib's
classification of noncompact Lorentz-isometric flows on compact
3-dimensional manifolds \cite[Thm. 2]{zeghib.lorentz.3d} gives two
possibilities for $(M,g')$:

\begin{enumerate}
 \item $(M,g')$ is flat and complete---that is, a compact quotient of Minkowski space.
 \item $M \cong G/ \Gamma_{\rho}$ for $G$ a finite
   cover of $\mbox{PSL}_2(\RR)$, and $\Gamma_{\rho}$
  the image of a uniform lattice
$\Gamma < G$ under a homomorphism $\mbox{Id}_\Gamma
\times \rho$ into $G\times G $; the image of
$\rho$ is in a 1-parameter hyperbolic or unipotent subgroup $\{
h^t \} < G$. The metric $g'$
lifts to a $G \times \{ h^t \}$-invariant metric on $G$.
 \end{enumerate}

It is a general fact that the conformal group of a flat, complete, Lorentzian manifold $(M,g')$ is inessential.
 Indeed, any $f \in \mbox{Conf}(M)$, can be lifted to a conformal transformation of Minkowski space $\RR^{1,2}$, 
 namely, an element of $\mbox{Sim}(\RR^{1,2}) \cong (\RR^* \times \mbox{SO}(1,2))  \ltimes
\RR^3$. It follows that $f$ is a \emph{homothetic} transformation,
one for which the conformal distortion is a constant
$\lambda$. But a homothety on a compact
 manifold is necessarily an isometry (consider the formula $\int_M
 \mbox{dvol}_{g'}=\int_M \mbox{dvol}_{f^*g'}$).
 
To simplify the argument for case (2), we initially assume that $G=\mbox{SL}_2(\RR)$.
Denote by $\Lambda$ the kernel 
of $\rho: \Gamma \to \{ h^t \}$, and by $\Lambda^{Z}$ the Zariski closure in $G$. 
 Note that $\Lambda$ is not solvable, because $\Gamma$, which is commensurable 
to a surface group, is not solvable. It follows that
$\Lambda^{Z}=G$. Denote by $\tilde{M}$ the cover of $M$ diffeomorphic
to $G$.  Because $G$ acts isometrically on the left on $\tilde{M}$, it acts
linearly on the finite-dimensional vector space $V$ of 
global conformal Killing fields of $\tilde{M}$. This representation is
given by an algebraic homomorphism $\alpha: \mbox{SL}_2(\RR) \to \mbox{GL(V)}$.  
Denote by $\tilde{\mathcal{X}}^{conf}(\tilde{M})$ the subspace of $V$
comprising lifts of vector fields in
$\mathcal{X}^{conf}(M)$. The restriction of $\alpha(\Lambda)$ to this
subspace is trivial, hence the same holds for $\Lambda^{Z}=G$. 

When $G$ is a quotient or a connected finite cover of $\mbox{SL}_2(\RR)$, the
previous arguments are easily adapted:
we lift or project $\Lambda$ to $\mbox{SL}_2(\RR)$, as appropriate. The representation of $G$ on $\tilde{\mathcal{X}}^{conf}(\tilde{M})$
 lifts to, or factors through, a representation $\alpha$ of
 $\mbox{SL}_2(\RR)$, for which the subgroup corresponding to $\Lambda$
 is trivial on $\tilde{\mathcal{X}}^{conf}(\tilde{M})$.
 The same holds for the Zariski closure, hence, this
 subspace is a trivial summand of $\alpha$.  We conclude that $G$
 centralizes $\tilde{\mathcal{X}}^{conf}(\tilde{M})$.
 
 Now $G$ commutes with all lifts of elements $h \in \mbox{Conf}^0(M)$.
Let $\tilde{h}$ be such 
a lift.
Choose a lift $\tilde{x}_0$ of $x_0$ to $\tilde{M}$, and let $\tilde{h}^* g'_{\tilde{x}_0} = \lambda g'_{\tilde{h}.\tilde{x}_0}$.  Given $\tilde{x} \in \tilde{M}$, let $f
\in G$ with $f.\tilde{x}_0 = \tilde{x}$.  As $f$ commutes with $\tilde{h}$,
$$ \tilde{h}^* g'_{\tilde{x}} = \tilde{h}^* f^* g'_{\tilde{x}_0} =
f^* \tilde{h}^*
g'_{\tilde{x}_0} = f^* \lambda g'_{\tilde{h}.\tilde{x}_0} = \lambda
g'_{f\tilde{h}.\tilde{x}_0} = \lambda g'_{\tilde{h}.\tilde{x}}$$

Thus $\tilde{h}$ is a homothety, and so is $h$. Because $M$ is
compact, $h$ must be an isometry.
We conclude that $\mbox{Conf}^{\,0}(M)$ is inessential.
\end{preuve}

\begin{remarque}
 We expect proposition \ref{prop.existence.essential.flow} to hold in
 any dimension, but this fact would obviously require a more general proof.
\end{remarque}

By proposition \ref{prop.existence.essential.flow}, under the assumptions of theorem
\ref{thm.main.theorem}, there is an essential conformal vector field
on $M$.  We fix such a vector field and call it $X$.  Because we
assume $(M,[g])$ to be real-analytic, so will $X$ be real-analytic.
In what 
 follows, we will work with the geometric structure defined by the
 pair $([g],X)$.

\section{Local and infinitesimal symmetries of $([g],X)$}
\label{sec.symmetries}

The conformal structure $[g]$ on $M$ determines a {\it rigid geometric structure of algebraic type}, in the sense of Gromov 
(see \cite{gromov.rgs}).
 It is also fruitful to consider the canonically associated Cartan geometry,
modeled on the 3-dimensional Lorentzian Einstein space
$\mbox{Ein}^{1,2}$. The latter space can be obtained as
$$\mbox{Ein}^{1,2} = (S^1
\times S^2/\langle \iota \rangle,[-d\theta^2 \oplus g_{S^2} ])$$
where
$\iota$ is the antipodal map on both factors.  It is a conformally
homogeneous space $\mbox{PO}(2,3)/P$, for $P$ the stabilizer of a
null line in $\RR^{2,3}$.   Denote $G = \mbox{PO}(2,3)$ with
  corresponding Lie algebra $\lieg$. The
Cartan geometry comprises (see \cite{sharpe} Ch V, \cite{cap.slovak.book.vol1} Sec 1.6):

\begin{itemize}
\item a principal $P$-bundle $\pi: \hat{M} \rightarrow M$; and 
\item a Cartan connection $\omega \in \Omega^1(\hat{M},\lieg)$ satisfying, for all $\hat{x} \in \hat{M}$,
\begin{enumerate}
 \item $\omega_{\hat{x}} : T_{\hat{x}} \hat{M} \rightarrow \lieg$ is a linear isomorphism
 \item $\omega_{\hat{x}.g} \circ R_{g*} = \mbox{Ad }g^{-1} \circ \omega_{\hat{x}} \ \ \forall g \in P$
 \item $\omega \left( \left. \frac{d}{dt} \right|_0 (\hat{x}.e^{tY}) \right) \equiv Y \ \  \forall \ Y \in \liep$
\end{enumerate}
\end{itemize}

The pair
$([g],X)$, with $X$ as in section \ref{sec.existence}, is also a rigid geometric structure of algebraic type. 
 It is not quite a Cartan geometry, but rather an {\it enhanced}
Cartan geometry, a notion which was studied in \cite[Sec
4.4.1]{pecastaing.frobenius}. It is proved there that the properties
of the local orbit structure which we will use 
are the same for enhanced Cartan geometries as for usual Cartan geometries.

\subsection{Local transformations and vector fields}

The conformal group of $M$ lifts to a group of automorphisms of the
Cartan bundle
$\hat{M}$ preserving $\omega$. Because $\omega$ gives a parallelization of
$\hat{M}$, the action of $\mbox{Conf}(M)$ on $\hat{M}$ is free and
proper (see, eg, \cite[Thm. 3.2]{kobayashi.transf}).

We will denote by $\mbox{Conf}^{loc}(M)$ the pseudogroup of local conformal
transformations of $M$.  As for global conformal transformations, any element of $\mbox{Conf}^{loc}(M)$ defined on some open subset 
$U \subset M$ lifts to an embedding of $\pi^{-1}(U)$
into $\hat{M}$ commuting with the principal $P$-action and preserving 
$\omega$.  We will work
below with the sub-pseudogroup $Z_X^{loc} \subset
\mbox{Conf}^{loc}(M)$ centralizing
$X$, where defined.

Let $\mbox{Is}^{loc}(x) \subset \mbox{Conf}^{loc}
 (M)$ be the stabilizer of a point $x \in M$; it is a group.  Any choice of $\hat{x} \in
\pi^{-1}(x)$ gives a monomorphism $\iota_{\hat{x}} :
\mbox{Is}^{loc}(x) \rightarrow P$, the \emph{isotropy monomorphism
  with respect to $\hat{x}$},
defined implicitly by
$$ h.\hat{x} =
\hat{x}.\iota_{\hat{x}}(h)$$
We will denote the image $\hat{I}_{\hat{x}}$.  A different choice $\hat{x}' = \hat{x}.p$
gives $\hat{I}_{\hat{x}'} = p^{-1} \hat{I}_{\hat{x}} p$.
  Denote $\mbox{Is}_X^{loc}(x)$ the subgroup of
$\mbox{Is}^{loc}(x)$ centralizing $X$.  For $\hat{x} \in \pi^{-1}(x)$, denote
$(\hat{I}_X)_{\hat{x}}$ the image of $\mbox{Is}^{loc}(x)$ under
$\iota_{\hat{x}}$.


A theorem of Amores \cite{amores.killing} says that on real-analytic
manifolds, germs of local
conformal vector fields can be uniquely extended along paths.  It follows that
the algebra of germs of local conformal vector fields defined around a
point $x \in M$ is independent of $x$, up to isomorphism.  Moreover,
local conformal vector fields on the universal cover of $M$ extend
 to global ones; note that these may not necessarily be complete.

 We will work below with the local conformal vector fields commuting
 with $X$.  Amores' theorem also implies that these form a
 well-defined subspace of the local conformal vector fields on $M$,
 which we will denote $\liezx$.

\subsection{Gromov's Frobenius theorem and isotropy}
\label{sec.gromov.frobenius}

Under the assumption that $M$ is compact and $C^\omega$, Gromov's
Frobenius theorem \cite{gromov.rgs}  ensures that, at each point $x$, a finite number of infinitesimal
conditions are sufficient for the production of local conformal
transformations at $x$. In the setting of analytic Cartan geometries, the
second author showed that the jets of the curvature $\kappa$ (see section 
\ref{sec.cartan-geometry-curvature} for the definition) provide this
sufficient condition \cite{me.frobenius}. The jet of order $i$  
can be captured by a $P$-equivariant map $D^{(i)} \kappa: \hat{M} \rightarrow
{\mathbb U}^{(i)}$, where ${\mathbb U}^{(i)}$ is a finite-dimensional vector space derived from the
curvature module, on which $P$ acts linearly. For the enhanced Cartan
structure $([g],X)$, there are a corresponding curvature $\kappa_X$
 and corresponding $P$-equivariant maps $D^{(i)} \kappa_X: \hat{M} \rightarrow
{\mathbb U}^{(i)}_X$ \cite[Sec 4.4.1]{pecastaing.frobenius}. In this setting,
the Frobenius theorem says:

\begin{theoreme} \cite[Prop 3.8]{me.frobenius}, \cite[1.6.C and 1.7.A]{gromov.rgs}
 \label{thm.frobenius}
 Let $(M,g)$ be a compact, real-analytic, pseudo-Riemannian manifold. Let $X$ be an analytic vector field on $M$. There exists an integer
  $i_0 \in \NN$ such that two points $x$ and $y$ are related by an
  element of $\mbox{Conf}^{loc}(M)$
 if and only if $D^{(i_0)}
  \kappa(\hx)=D^{(i_0)} \kappa(\hy)$ for some $\hx \in \pi^{-1}(x)$ and $\hy \in \pi^{-1}(y)$.
  The same holds for $Z_X^{loc}$ with $\kappa_X$ in place of $\kappa$.
\end{theoreme}

From theorem \ref{thm.frobenius}, the group $\hat{I}_{\hat{x}}$ coincides with  
 the stabilizer of $D^{(i_0)} \kappa(\hx)$ for the representation of
 $P$ on ${\mathbb U}^{(i_0)}$; similarly, $(\hat{I}_X)_{\hat{x}}$ is
 the stabilizer of $D^{(i_0)} \kappa_X(\hx)$.
This leads to the following result on the structure of isotropy groups
(which in fact holds more generally, for
real-analytic rigid geometric
structures of algebraic type).

\begin{theoreme}[Gromov \cite{gromov.rgs} 3.4.A ]
  \label{thm.isotropy.algebraic}
Let $(M,g)$ be a real-analytic pseudo-Riemannian manifold, and let $X$ be a real-analytic vector field.
  With respect to any $\hat{x} \in \pi^{-1}(x)$, the isotropy images
$\hat{I}_{\hat{x}}$ and $(\hat{I}_X)_{\hat{x}}$
  are real-algebraic subgroups of $P$; in particular,
  $\mbox{Is}^{loc}(x)$ and $\mbox{Is}_X^{loc}(x)$ have finitely-many
  components.
\end{theoreme}

\subsection{Local orbit stratification}
\label{sec.stratification}

In this section, we will
focus on the structure of the $\mbox{Conf}^{loc}(M)$-orbits in $M$. By the $\mbox{Conf}^{loc}(M)$-orbit of a point $x
\in M$, we mean all points that can be reached from $x$ by applying a finite
sequence of local conformal maps; it will be denoted
$\mathcal{O}(x)$.
The $Z_X^{loc}$-orbit
 of the point $x$ is defined analogously, and denoted $\mathcal{O}_X(x)$.
 
Here are consequences of Gromov's stratification theorem which will be
used below (see also \cite[Thm 4.1]{me.frobenius}, \cite[Thm
4.19]{pecastaing.frobenius}). This theorem stems from the
 Frobenius theorem \ref{thm.frobenius} and properties of orbits for algebraic actions.

  \begin{theoreme}[Gromov \cite{gromov.rgs} 3.1.A, 3.2] Let $(M,g)$ be a compact
    real-analytic pseudo-Riemannian manifold.
    \label{thm.gromov.stratification}
    \begin{enumerate}
    \item For all $x \in M$, the orbit $\mathcal{O}(x)$ is a semi-analytic subset of $M$. It is locally
      closed and has finitely-many components.  The same holds for $\mathcal{O}_X(x)$.
      \item For all $x \in M$, the closure
        $\overline{\mathcal{O}(x)}$ is locally connected and contains a closed
      $\mbox{Conf}^{loc}(M)$-orbit.  The
      analogous properties hold for $\overline{\mathcal{O}_X(x)}$.
        \end{enumerate}
\end{theoreme}

Recall that a subset $S$ of a topological space is \emph{locally
  closed} if $S$ is open in the closure $\overline{S}$.  A set is 
\emph{semianalytic} if it is locally cut out by
finitely many analytic equalities and inequalities; see \cite{bierstone.milman} for
properties of these sets.  The closure of a semianalytic set is again
semianalytic.
Local connectedness of semianalytic sets can be found in \cite[Cor 2.7]{bierstone.milman}.

In a slight abuse of language, we will call the \emph{$\liezx$-orbit}
of a point $x$ the set of points reachable from $x$ by flowing along
finitely many local vector fields in $\liezx$.
 A consequence of the proof
 of theorem \ref{thm.gromov.stratification} is that the $\liezx$-orbit of $x$ is the connected component
  of $\mathcal{O}_X(x)$ containing $x$.  We will implicitly make this
  identification several times below.

\subsection{Recurrence produces isotropy}
\label{sec.recurrence}

In \cite{frances.open.dense} and \cite{frances.lorentz.3d}, the first
author combined the Frobenius theorem \ref{thm.frobenius} with Poincar\'e
recurrence to produce nontrivial local isotropy for isometric actions. 
Recall that a \emph{recurrent} point for an unbounded subgroup $H < \mbox{Conf}(M,[g])$ is $x \in M$ with
$h_k.x \to x$ for some unbounded sequence $\{ h_k \} \subset H$.

\begin{proposition}(compare \cite[Prop. 5.1]{frances.open.dense},
  \cite[Prop. 3.3]{frances.lorentz.3d})
  \label{prop.recurrence}
  Let $(M,g)$ be a compact, real-analytic
  pseudo-Riemannian manifold, and let $\{ \varphi_X^t \}$ be a
  noncompact conformal flow.  At each recurrent 
  point $x$ for $\{ \varphi_X^t \}$, the local isotropy $\mbox{Is}_X^{loc}(x)$
has noncompact identity component.
  \end{proposition}


\begin{preuve}
 Let $x \in M$ be  recurrent for $\{ \varphi_X^t \}$, and consider $\hx \in \pi^{-1}(x)$. 
  There are $t_k \rightarrow \infty$ and $\{ p_k \}
  \subset P$
  such that $\varphi_X^{t_k}.\hx.p_k^{-1} \to \hx$.
    Because $\mbox{Conf}(M)$
    acts properly on $\hat{M}$, the sequence $\{ p_k \}$ necessarily
    tends to infinity.  Let $D^{i_0}\kappa_X: \hat{M} \to {\mathbb U}^{(i_0)}$ be the
     $P$-equivariant  map given by theorem \ref{thm.frobenius}.  This
     map is also $\mbox{Conf}(M)$-invariant, so that
     $p_k.D^{(i_0)}\kappa_X(\hat{x}) \rightarrow D^{(i_0)}
\kappa_X(\hat{x})$.  Now $P$ acts algebraically on ${\mathbb
  U}^{(i_0)}$ with locally closed orbits, which implies 
$p_k.D^{(i_0)}\kappa_X(\hat{x})= \epsilon_k.D^{(i_0)}\kappa_X(\hat{x})
$ for some  $\epsilon_k$ tending to the identity in $P$.
 This implies existence of a
noncompact stabilizer of $D^{(i_0)} \kappa_X(\hat{x})$, coinciding with 
$(\hat{I}_X)_x$, again by theorem \ref{thm.frobenius}. Because this stabilizer
is moreover algebraic in $P$,  it has noncompact identity component.
\end{preuve}

\section{Conformal curvature and vanishing conditions}
\label{sec.curvature.vanishing}
In this section, we gather several sufficient conditions for conformal
flatness, which will be applied throughout our proof.  Although some
definitions and results later in this section will be valid in higher
dimensions, we assume for now that $(M,g)$ is a 3-dimensional, smooth
Lorentzian manifold.

\subsection{Cotton-York tensor}
  \label{sec.cotton.tensor}
  
  Recall that in dimension 3,
  the Weyl curvature vanishes, and the obstruction
  to conformal flatness is the Cotton-York tensor  $C \in \Gamma(\wedge^2
  T^*M \otimes T^*M)$ (see \cite[II.28]{eisenhart.riem.geom}):
  $$ C_x(u,v,w) = (\nabla_w P)(u,v) - (\nabla_v P)(u,w)$$
for $u,v,w \in T_xM$, where $P$ is the Schouten tensor
$$ P_x(u,v) = \mbox{Ric}_x(u,v) - \frac{1}{4} \mbox{Sc}(x) g_x(u,v)$$
and $\mbox{Ric}$ and $\mbox{Sc}$ denote the Ricci and scalar
curvatures of $g$, respectively.  The Cotton-York tensor is
conformally invariant, meaning it is independent of a choice of metric
in the conformal class $[g]$.  For $f \in \mbox{Conf}(M)$, in
particular,
$$ C_{f(x)} (f_{*x}u,f_{*x}v,f_{*x}w) = C_x (u,v,w) \qquad \forall x \in M, \ u,v,w \in T_xM$$
This tensor moreover satisfies the Bianchi identity, 
meaning it is in the kernel of the map to $\wedge^3 T^*M$, and is
totally trace-free.


The Cotton-York tensor is a section of the vector bundle associated to the following module $\mathbb{U}$.
Choose
$$\mathbb{I} =
\left( \begin{array}{ccc}
         & & 1\\
         & 1 &  \\
         1 &  &
       \end{array}
     \right)
     $$
     and write $\mbox{SO}(1,2)$ for $\mbox{SO}(\mathbb{I})$.
Write $E_1, E_2, E_3$ for the standard basis of $\RR^3$, and take $\xi_i =
E_i^t \mathbb{I}$ as basis for $\RR^{3*}$, $i = 1, 2,3$.  Denote
$$\mathbb{U} = (\wedge^2 \RR^{3*} \otimes \RR^{3*})^C$$
the 5-dimensional $\mbox{SO}(1,2)$-module of trace-free tensors satisfying the Bianchi
symmetry.  It is the sum of five 1-dimensional weight spaces, for
weights $w = -2, -1, 0, 1,2$.  The positive weight spaces are
$$ \mathbb{U}^{+2} = \RR(\xi^1 \wedge \xi^2 \otimes \xi^1) \qquad \mathbb{U}^{+1} = \RR(\xi^1 \wedge \xi^3 \otimes \xi^1 - \xi^1 \wedge \xi^2 \otimes \xi^2) $$ 
Denote $\mathbb{U}^+ = \mathbb{U}^{+1} + \mathbb{U}^{+2}$.  Denote $G_0 = \RR^* \times \mbox{SO}(1,2) = \mbox{CO}(1,2)$.  Letting $d \in \RR^*$
act by the scalar $d^{-3}$ extends the representation on $\mathbb{U}$
to $G_0$.



\subsection{Cartan curvature}
\label{sec.cartan-geometry-curvature}

Let $(M,\hat{M},\omega)$
be the canonical Cartan geometry modeled on $\mbox{Ein}^{1,2}$
associated to $(M,[g])$, from section \ref{sec.symmetries}.
Recall the notation $G = \mbox{PO}(2,3)$.

%


The \emph{Cartan curvature}
$$\Omega(X,Y) = d\omega(X,Y) + [\omega(X),\omega(Y)]$$
is a semi-basic 2-form on $\hat{M}$.  It is the obstruction to
$(M,\hat{M},\omega)$ being locally isomorphic to the model Cartan
geometry, which in our case is
$(\mbox{Ein}^{1,2},G,\omega_G)$, where $\omega_G$ is the Maurer-Cartan
form of $G$.
The values of $\Omega$ lie in the nilpotent radical
$\liep^+$ of $\liep$, which is identified as a $G_0$-representation
with $\RR^{1,2*}$.

Via the Cartan
connection $\omega$, the curvature $\Omega$ can be identified with a
$P$-equivariant function
$$ \kappa : \hat{M} \rightarrow \wedge^2 (\lieg / \liep)^* \otimes
\liep^+$$
where the $P$-representation on the target is built from the adjoint
representation of $G$ restricted to $P$.  It factors through $P/P^+
\cong G_0$.  As $G_0$-modules,
$$\wedge^2 (\lieg / \liep)^* \otimes
\liep^+ \cong \wedge^2 \RR^{1,2*} \otimes \RR^{1,2*}$$
The Cartan connection identifies $T^*M$ with $\hat{M} \times_P
\liep^+$.  In this way, $\kappa$ corresponds to the Cotton
tensor, read in the conformal frames given by $\hat{M}/P^+$ (see
\cite[Cor 1.6.8]{cap.slovak.book.vol1}).  In
particular, $\kappa$ factors through a $G_0$-equivariant map to $\mathbb{U}$.

\subsection{Holonomy sequences}

Because $M$ is compact, while local conformal transformations act
freely and properly on $\hat{M}$, we can associate to an unbounded
sequence in $\mbox{Conf}^{loc}(M)$ an unbounded sequence in the
principal group $P$.

\begin{definition}
  Let $\hat{x} \in \hat{M}$ and $\{f_k \} \subset
  \mbox{Conf}^{loc}(M)$, with $f_k$ defined on a neighborhood of $x=
  \pi(\hat{x})$ for all $k$.  A
\emph{holonomy sequence for $\{ f_k \}$ at $x$ with respect to $\hat{x}_k \rightarrow
  \hat{x}$},
is $\{ p_k \} \subset P$ such that
$$ f_k.\hat{x}_k.p_k^{-1} \rightarrow \hat{y} \qquad \mbox{for some } 
\hat{y} \in \hat{M}$$
A \emph{pointwise holonomy sequence} is $\{ p_k \}$ as above for which
$\hat{x}_k \in \pi^{-1}(x) \ \forall k$.
\end{definition}

Note that given a holonomy sequence as above, $f_k.x_k \rightarrow y$, where $x_k =
\pi(\hat{x}_k), y = \pi(\hat{y})$.  Given any sequence $\{ f_k \}
\subset \mbox{Conf}^{loc}(M)$ and any $x_k \rightarrow x \in M$, we may assume, after passing to a
subsequence, that $f_k.x_k $ converges in $M$, since $M$ is compact.  Then
there are holonomy sequences, including pointwise holonomy sequences, for $\{f_k \}$ at $x$.

A conformal transformations is, in dimension at least 3, determined by its $2$-jet at 
 a point (see, eg, \cite[Sec IV.6]{kobayashi.transf}).  
A holonomy sequence captures the $2$-jets of $\{f_k \}$
 along the sequence $\{ x_k \}$ and thus turns out to be a useful tool to understand the local behavior of $\{f_k\}$ around $\{x_k \}$.

Since the $2$-jet of a map can be read in different $2$-frames, holonomy sequences are far from unique.  Some of the choices involved
in their construction are accounted for by vertical equivalence: two holonomy sequences $\{ p_k \}$,
$\{ q_k \}$ are \emph{vertically equivalent} if $q_k = a_k p_k b_k$ for $\{
a_k \}, \{ b_k \} \subset P$ bounded.

\subsubsection{Taxonomy of holonomy sequences}

The reductive group $G_0 \cong \mbox{CO}(1,2)$ has a $KAK$
decomposition, in which the $\RR$-split Cartan subgroup $A$ is
two-dimensional.  Under the embedding of $G_0$ in
$G=\mbox{PO}(2,3)$ as the stabilizer of a pair of nonorthogonal isotropic
lines in $\RR^{2,3}$, the torus $A$ equals the $\RR$-split Cartan subgroup of
$G$.  Up to vertical equivalence, a holonomy sequence
in $P$ can be written $p_k = d_k \tau_k$, with
$d_k \in A$ and $\tau_k \in P^+$.

Denote by $\liea$ the subalgebra corresponding to $A$.  The
standard choice of simple roots spanning $\liea^*$ comprises a long
root, $\gamma$, and a short root, $\beta$.  The latter vanishes on the
$\RR$-factor in the decomposition $\lieg_0 = \RR \oplus
\mathfrak{o}(1,2)$, and it can be thought of as the generator of the root
space of $\mathfrak{o}(1,2)$.  We take $\alpha = \gamma - \beta$; it
corresponds to the negative log conformal dilation in the
standard representation on $\RR^{1,2}$.  Explicitly, for the quadratic
form $2 x_0x_4 + 2 x_1 x_3 + x_2^2$ on $\RR^{2,3}$,

$$ 
\liea =
\left \{
\begin{pmatrix}
a & & & & \\
  & b & & & \\
  & &  0 & & \\
  & & & -b & \\
  & & & & -a
\end{pmatrix}
,\ a,b \in \RR 
\right \}
$$

with $\alpha$ and $\beta$ dual to the $a$ and $b$ parameters, respectively.

The symmetry $(\alpha,\beta) \mapsto (\alpha, - \beta)$ can be
realized by conjugation in $P$.  Denote $A'$ the semigroup comprising
all $d \in A$ with $\beta(\ln d) \leq 0$.  Up to vertical equivalence,
a holonomy sequence in $P$ may be assumed to be in $A'P^+$.
\begin{definition}
\label{def:holonomy_trichotomy}
  Let $\{ d_k \tau_k \}$ be an unbounded sequence in $A'P^+$, and let $D_k = \ln d_k$.  The sequence is said to be
\begin{itemize}
\item of \emph{bounded distortion} if $\alpha(D_k)$ is bounded while
  $\beta(D_k) \rightarrow - \infty$
\item \emph{contracting}
  if $(\alpha + \beta)(
  D_k) \rightarrow \infty$.
\item \emph{balanced} if $\alpha(D_k) + \beta(D_k)$ is bounded, but each term is unbounded.
\item \emph{mixed} if $\beta(D_k) \rightarrow - \infty$ and
  $\alpha(D_k) \rightarrow \infty$, while $(\alpha + \beta)(D_k)
  \rightarrow - \infty.$ 
\end{itemize}
It is called \emph{linear} if $\tau_k \equiv 1$.

  \end{definition}

  \subsubsection{Stability and propagation of holonomy}
  \label{subsec.stability.propagation}
 The following definition is inspired by \cite{zeghib.tgl1} (see also \cite[Sec
7.4]{dag.rgs}, \cite[Def 2.10]{cap.me.parabolictrans} for a
non-approximate version, and \cite[Sec 4.4]{frances.degenerescence}
for a related notion of stability and stable foliations).

\begin{definition}
Let $\mathbb{V}$ be a $P$-module, and let $\{ p_k \}$ be a sequence in $P$.  The \emph{approximately stable set} for $\{ p_k \}$ in $\mathbb{V}$ is
$$ \mathbb{V}^{AS}(p_k) = \{ v = \lim v_k \in \mathbb{V} \ : \ p_k.v_k \ \mbox{is bounded} \} $$
\end{definition}

The following proposition is a version for sequences of \cite[Prop
2.9]{cap.me.parabolictrans} (see \cite[Prop 3.13]{mp.confdambra} for the one-line
proof):

\begin{proposition}
\label{prop:as_for_inv_sections}
Given a $P$-module $\mathbb{V}$, represent a continuous, $\{ f_k \}$-invariant section of the associated bundle $\hat{M} \times_P \mathbb{V}$ by a continuous, $P$-equivariant, $\{ f_k \}$-invariant map $\sigma : \hat{M} \rightarrow \mathbb{V}$.  Given any holonomy sequence $\{ p_k \}$ for $\{ f_k \}$ with respect to $\hat{x}_k \rightarrow \hat{x}$, the value $\sigma(\hat{x}) \in \mathbb{V}^{AS}(p_k)$.
\end{proposition}

  An unbounded sequence $\{ d_k \tau_k \}$ in $A'P^+$ is called \emph{stable} if it is linear 
   and contracting or balanced.  In general, holonomy sequences for a
   given $\{f_k \} \subset \mbox{Conf}^{loc}(M)$ 
    can be of different types at different nearby points; however, stable
    sequences enjoy the property of local {\it propagation of holonomy}.
    
  The exponential map of the Cartan connection is the vehicle for
  propagation of holonomy.
   Any $X \in \lieg$ defines a vector field $\hat{X}$ on $\hat{M}$ by
$\omega(\hat{X}) \equiv X$. Denote the flow along $\hat{X}$ by $\{ \varphi^t_{\hat{X}} \}$. 
The \emph{exponential map} at $\hat{x} \in \hat{M}$ is 
$$ \exp_{\hat{x}}(X) = \varphi^1_{\hat{X}}(\hat{x}) \in \hat{M}$$
for $X$ in a sufficiently small neighborhood of $0$ in $\lieg$.
%
%
   A holonomy sequence for $\{ f_k \}$ at $x$ turns out to be also valid along certain
exponential curves from $x$ (see \cite[Prop 6.3]{frances.locdyn} or \cite[Prop 3.9]{mp.confdambra}):

\begin{proposition}
\label{prop:propagation_holonomy}
 Let $\{p_k \}$ be a holonomy sequence for $f_k$ at $x$, 
 with respect to $\hat{x}_k \in \pi^{-1}(x_k)$.  
 Suppose given $Y_k  \rightarrow Y \in \lieg \backslash \liep$ for which $\Ad p_k(Y_k)$ converges.  
 Then, provided $Y$ is in the domain of $\exp_{\hat{x}}$, $\{ p_k \}$ is also a holonomy sequence for $\{ f_k \}$ at $x' = \pi \circ \exp(\hat{x},Y)$ with respect to $\hat{x}_k' = \exp(\hat{x}_k,Y_k)$.  
\end{proposition}

 One derives  easily from proposition \ref{prop:propagation_holonomy} the following corollary.
 
 \begin{corollaire}\cite[Lem 4.3, 4.6]{frances.degenerescence}
  \label{coro.stable}
  If $\{f_k \}$ has a
 stable holonomy sequence $\{p_k \}$ at $x$, then $\{ p_k \}$ is also
 a holonomy sequence for $\{ f_k \}$ on a neighborhood of $x$. 
 \end{corollaire}


\subsection{Stability implies conformal flatness}
\label{sec.stability-flatness}
We record some immediate consequences of the
properties outlined above.

\begin{proposition}
\label{prop.balanced.flat}
If there is a balanced or a contracting holonomy sequence at $x \in M$, then the Cotton-York
tensor vanishes at $x$.  
\end{proposition}

\begin{proof}
Let $\{ p_k = d_k \tau_k \}$ be a holonomy sequence at $x$, with $\{
d_k \} \subset A'$ satisfying the balanced or contracting condition in definition
\ref{def:holonomy_trichotomy}.  The $P$-representation on the Cotton
module $\mathbb{U}$ factors through the projection to $G_0$, which has
weights $3 \alpha + w \beta$, $w = -2,-1,0,1,2$.  
Now
$p_k$ acts by $d_k$, for which $(3 \alpha + w \beta)(\ln d_k)
\rightarrow \infty$, for all possible $w$.  Thus $\mathbb{U}^{AS}(p_k)
= 0$.  The conclusion follows from proposition \ref{prop:as_for_inv_sections}.
  \end{proof}

Assuming $(M,g)$ is not conformally flat, we may thus assume that
the set of points admitting a balanced or a contracting holonomy sequence is nowhere
dense.

\begin{proposition}
\label{prop.stable.flat}
 Assume that $(M,g)$ is real-analytic.  If there exists an unbounded sequence $\{ f_k \} \subset \mbox{Conf}^{loc}(M)$, all defined on
a neighborhood $U$ of $x \in M$, admitting a stable 
  holonomy sequence at $x$, then 
  $(M,g)$ is conformally flat.
\end{proposition}

\begin{preuve}
  A stable, unbounded holonomy sequence $\{p_k = d_k \tau_k \}$ is balanced or
  contracting.  As
  remarked above, proposition \ref{prop:propagation_holonomy} implies
  that all points in a neighborhood $V \subseteq U$ of $x$ admit the
  same stable holonomy sequence.  By proposition
  \ref{prop.balanced.flat}, the Cotton-York tensor vanishes on $V$   
 (see also \cite[Prop 5]{frances.ccvf}).  By the analyticity assumption, $(M,g)$ is conformally flat everywhere.
\end{preuve}

\subsection{Big isotropy implies conformal flatness}
\label{sec.big-isotropy}
Next we recall a key linearization theorem for conformal vector fields:

\begin{theoreme}[Frances--Melnick \cite{fm.champsconfs} Thm 1.2]
\label{thm.fm.linearization}
  Let $(M,g)$ be
  a real-analytic Lorentzian manifold of dimension at least 3.  Let $x \in M$ and let $X \in
  \mathcal{X}^{conf}(M)$ vanish at $x$, with local flow $\{
  \varphi^t_X \} < \mbox{Is}^{loc}(x)$.  If for some $\hat{x} \in \pi^{-1}(x)$, the image
  $\iota_{\hat{x}}(\{ \varphi^t_X \}) < G_0$, then
  $\{ \varphi^t_X \}$ is linearizable in a neighborhood of $x$.
  Otherwise, $(M,g)$ is conformally flat.
    \end{theoreme}

Combining the theorem above with the previous propositions yields two
useful corollaries.

\begin{corollaire}
 \label{cor.small-stabilizer}
Assume that $(M,g)$ is real-analytic.  If for some $x \in M$,
  the isotropy algebra $\mathfrak{Is}^{loc}(x)$ contains a $2$-dimensional abelian subalgebra, then $(M,g)$ is conformally flat.
\end{corollaire}

\begin{preuve}
Consider the linear part of the isotropy:
given $\hat{x} \in \pi^{-1}(x)$, compose the isotropy homomorphism 
$(\iota_{\hat{x}})_{*e}: \mathfrak{Is}^{loc}(x) \to \liep$
with the projection to $\lieg_0$,
to obtain $\lambda_{\hx}: \mathfrak{Is}^{loc}(x) \to \lieg_0$. 
Because $(M,g)$ is real analytic, if $\lambda_\hx$ is not
  injective, then $(M,g)$ is conformally flat by theorem \ref{thm.fm.linearization}.
Assuming $\lambda_{\hx}$ is injective, it has $2$-dimensional, abelian
image in $\lieg_0 \cong \RR \oplus \oo(1,2)$, 
necessarily containing the center.  There is thus a linear, contracting
1-parameter subgroup of $\mbox{Is}^{loc}(x)$.
    Proposition 
    \ref{prop.stable.flat} implies in this case that $(M,g)$ is
    conformally flat.  
\end{preuve}

Now let $X \in \mathcal{X}^{conf}(M)$ and $x \in M$. Let $Z_X^{loc} \subset
\mbox{Conf}^{loc}(M)$, 
$\liezx$, and $\mbox{Is}_X^{loc}(x)$   be as in section \ref{sec.gromov.frobenius}. The  Lie
algebra of $\mbox{Is}_X^{loc}(x)$ is denoted $\mathfrak{Is}_{X}^{loc}(x)$.

\begin{corollaire}
\label{coro.isotropy-dim2}
Assume $(M,g)$ is real-analytic.  If the dimension of
  $\mathfrak{Is}_{X}^{loc}(x)$ is at least $2$ for some $x \in M$, then $(M,g)$ 
  is conformally flat.
\end{corollaire}

\begin{preuve}
 If $X(x)=0$, namely $X \in \mathfrak{Is}_{X}^{loc}(x)$, then 
 $\mathfrak{Is}_{X}^{loc}(x)$ contains a two-dimensional abelian subalgebra, and we conclude by Corollary \ref{cor.small-stabilizer}.

Otherewise, $\mathfrak{Is}_{X}^{loc}(x)$
  annihilates $X(x) \neq 0$.
Again by theorem \ref{thm.fm.linearization}, we can assume that the
image $(\iota_{\hat{x}})_{*e}(\mathfrak{Is}^{loc}(x))$ is in $\lieg_0$ for some
$\hat{x} \in \pi^{-1}(x)$.
 A $0$-eigenvector in
 $\RR^{1,2}$ of a 2-dimensional
 subalgebra of $\lieg_0$ is necessarily lightlike, and the annihilator
 of a lightlike vector has dimension exactly 2.  It includes a
 diagonal subgroup $\{ d^t = e^{tD} \}$ with $\alpha(D) = - \beta(D)$;
 any unbounded sequence $\{ d^{t_k} \}$ has balanced, linear holonomy with
 respect to $\hat{x}$.
 Proposition \ref{prop.stable.flat} thus applies and ensures that  $(M,g)$ is  conformally flat. 
   \end{preuve}

   \subsection{A new curvature vanishing result}

   We will need the following strengthening of \cite[Thm
   1.4]{fm.champsconfs}. The proof is somewhat shorter for
   3-dimensional manifolds, so we restrict to that case here.

   \begin{theoreme}
     \label{thm.translation.holonomy}
Let $(M,g)$ be a real-analytic, compact, 3-dimensional Lorentzian manifold.  Let $\{
f_k \} \subset \mbox{Conf}(M,[g])$ be an unbounded sequence.  If $\{ f_k \}$
admits a holonomy sequence at $x \in M$ contained in $P^+$, then
$(M,g)$ is conformally flat.
   \end{theoreme}

Notice that the hypothesis that $\{ f_k \}$
admits a holonomy sequence at $x \in M$ contained in $P^+$  is equivalent to both sequences of differentials 
$\{D_x f_k\}$ and $\{(D_x f_k)^{-1}\}$ being bounded. 
The proof of the theorem is somewhat technical, and is deferred to section \ref{sec.appendix}, at the end of the paper.

\section{Case $\liezx$ is $4$-dimensional}
\label{sec.4d}

  We begin the proof of the main theorem \ref{thm.main.theorem}. 
  Recall that $(M,g)$ is 3-dimensional, compact, real-analytic, and
Lorentzian. The group 
$\mbox{Conf}^{\,0}(M,[g])$ is assumed to be  essential.  By proposition \ref{prop.existence.essential.flow}, it admits an essential 
 conformal vector field $X \in
\mathcal{X}^{conf}(M)$.  Denote by $\liezx$ the algebra of local
conformal vector fields on $M$ commuting with $X$ (see section \ref{sec.gromov.frobenius}).  By corollary
\ref{coro.isotropy-dim2}, the dimension of $\liezx$ is at most four.  Of course, $X
\in \liezx$, so it has dimension at least one.  We will prove theorem \ref{thm.main.theorem} by analyzing each possible value 
of this dimension.

Suppose $\liezx$ has dimension four.  Corollary
\ref{coro.isotropy-dim2} implies that all $\liezx$-orbits have
dimension $3$, hence there is only one such orbit. In particular,
$(M,g)$ is locally conformally homogeneous and $X$ is nowhere vanishing.
Given $x \in M$, the isotropy  $\mbox{Is}_X^{loc}(x)$ 
 fixes the nonzero vector $X(x)$.  By theorem
 \ref{thm.fm.linearization}, $(M,g)$ is conformally flat, or there
 exists $\hat{x} \in \pi^{-1}(x)$ with $(\hat{I}_X)_\hx
 < G_0$.  This subgroup 
 is $1$-dimensional, and two cases may occur. 
 
 If $(\hat{I}_X)_\hx$ is not unimodular, then it is conjugate in $G_0$
 to a diagonal subgroup
 $ \{ \mbox{diag}(1, \lambda , \lambda^2) \ | \ \lambda \in \RR^*  \}.$
Any unbounded sequence of this group is balanced and linear.  
 Then proposition \ref{prop.stable.flat} ensures that $(M,g)$ is conformally flat. 

Next suppose $(\hat{I}_X)_\hx$ is unimodular.  Then there is $\lambda_x \in
[g]_x$ on $T_xM$ which is preserved by $D_xf$ for every $f \in
\mbox{Is}_X^{loc}(x)$.  Given $y \in M$, choose 
$f \in Z_X^{loc}$
 sending $x$ to $y$, and define $\lambda_y= f_*\lambda_x$. This does
 not depend on the choice of $f$, because for another choice, say 
 $h \in Z_X^{loc}$, the difference $h^{-1} \circ f \in
 \mbox{Is}_X^{loc}(x)$, which preserves $\lambda_x$.
The result is a metric $\lambda \in [g]$ which is $Z_X^{loc}$-invariant.
Note that $\lambda$ is analytic; indeed, given $y \in M$, there are $Y, Z, T \in \liez_X$
defined on a neighborhood of $y$,
with values at $y$ spanning $T_yM$.
 The map $(u,v,w) \mapsto \varphi_Y^u \circ \varphi_Z^v \circ \varphi_T^w (y)$ provides a local analytic chart around $y$, in which 
 $\lambda$ is analytic.  We conclude that $\{ \varphi_X^t \}$ is
 inessential, a contradiction.


\section{Case $\liezx$ is $3$-dimensional}
\label{sec.3d}

The center of $\liezx$ is nontrivial because it contains $X$. Thus if
$\liezx$ is 3-dimensional, it could be isomorphic to $\RR^3$, $\heis(3)$, or $\mathfrak{aff}(\RR) \oplus \RR$.
For the sake of efficiency, we will assume for the rest of this
section that $(M,g)$ is \emph{not} conformally flat, in order to arrive
at a contradiction with the fact that $\{ \varphi_X^t \}$ is essential.
The results collected thus far lead to:
\begin{proposition}
 \label{prop.propertiesdim3}
 \begin{enumerate}
 \item The flow $\{\varphi_X^t \}$ has no singularities.
 \item All $\liezx$-orbits have dimension at least two.
 \item There is a closed $\liezx$-orbit $\Sigma$, which is a torus or a Klein bottle, on which $X$ is lightlike.
 \end{enumerate}
\end{proposition}

 \begin{preuve}
  If there were a singularity $x$, the
   differential $\{ D_x\varphi_X^t \}$
 at $x$ would fix two linearly independent vectors, the
 values at $x$ of two elements of $\liezx$ linearly independent modulo
 $X$.  From the fact that the differential preserves $g_x$ up to
 scale, simple linear algebra leads to the conclusion that
 $D_x\varphi_X^t=\mbox{Id}_{T_x M}$ for all $t$, a contradiction with theorem \ref{thm.fm.linearization}. 
 
 Point $(2)$ follows from 
 corollary \ref{coro.isotropy-dim2}.
 
 If $g(X,X)$ were nonvanishing, then
$\{ \varphi_X^t \}$ woud preserve $g/g(X,X)$ and be inessential. Thus
the zero set $\Lambda$ of $g(X,X)$ is nonempty and closed. By theorem \ref{thm.gromov.stratification} (2), there exists $x \in \Lambda$
 such that the $\liezx$-orbit $\Sigma$ of $x$ is closed. If $\Sigma$
 is $3$-dimensional then it equals $M$, and the identity component of
 $\mbox{Is}_X^{loc}(y)$ is trivial at each $y \in M$.  But since $M$ is compact, 
  the flow 
$\{ \varphi_X^t \}$ has recurrent points, which leads to a contradiction with 
 proposition 
 \ref{prop.recurrence}.  Thus $\Sigma$ must be a closed surface.  It is a torus or a Klein bottle because $X$ is
   nonvanishing. By construction, $X$ is lightlike on $\Sigma$. Point
   (3) is proved.  \end{preuve}


Arguments follow for each possible isomorphism type of $\liezx$.

\subsection{Case $\liezx$ is isomorphic to $\RR^3$}

Let $\Sigma$ be a 2-dimensional orbit, as guaranteed by proposition
\ref{prop.propertiesdim3} (3), and let $x \in \Sigma$.  Because $\liezx$ is abelian,
the isotropy at $x$ fixes two linearly independent vectors, spanning
$T_x \Sigma$.  As in the proof of proposition
\ref{prop.propertiesdim3} (1), we have a contradiction with theorem
\ref{thm.fm.linearization}.

\subsection{Case $\liezx$ is isomorphic to $\heis(3)$}

In this case, $X$ generates the center of $\heis(3)$.

First suppose there exists an open $\liezx$-orbit $\Omega$. On a
sufficiently small open subset $U \subset \Omega$, there is
$g_0 \in \left. [g] \right|_{U}$ such that $(U,g_0)$ is isometric to an open subset of  $\Heis(3)$ endowed with a
 left-invariant Lorentzian metric.  
Left-invariant Lorentzian metrics on 
 $\Heis(3)$ were classified in \cite{rahmani.lorentz.heis}; there are
 three isometry types, according to the sign of $\langle X,X \rangle$.
 If $\langle X,X \rangle =0$, the metric is flat. By the analyticity
 assumption, $(M,g)$ is conformally flat, contradicting our current hypothesis.
If $\langle X,X \rangle \neq 0$, the isometry group of the metric on $\Heis(3)$ is
 $4$-dimensional and centralizes the center of $\Heis(3)$.  
  Then $\liezx$ has dimension at least 4, contradicting our current assumption that it is 3.
 %

We conclude that the $\liezx$-orbits are all $2$-dimensional.  The
following proposition, when combined with our analyticity assumption, concludes this case.

\begin{proposition}
\label{prop.heisenberg}
 Let $(M, g)$ be a smooth, $3$-dimensional,  Lorentzian manifold.  Suppose there is a
 nonempty open subset $\Omega \subseteq M$ with a local conformal
 action of $\heis(3)$, such that all pseudo-orbits are $2$-dimensional.  Then  $(M, g)$ is conformally flat.
\end{proposition}

\begin{preuve}
  Let $Y$ and $Z$ be the further generators of $\heis(3)$, such that $[X,Y]=[X,Z]=0$, and $[Y,Z]=X$. 
    Since
 $X$ and $Y$ commute and span a $2$-dimensional space at each point of
 $\Omega$, there exist local coordinates $(x,y,z)$
 in which $X=\delx$ and $Y=\dely$.
  Because the $\liezx$-orbits are $2$-dimensional, $Z$ is of the form $\lambda \delx + \mu \dely$ for some functions $\lambda$ and $\mu$.
    The bracket relations lead to 
    $$0=\frac{\partial \lambda}{\partial x}=\frac{\partial
      \mu}{\partial x}=\frac{\partial \mu}{\partial y} \qquad
    \mbox{and} \qquad 
    \frac{\partial \lambda}{\partial y}=1$$
    Hence we can write
     $$ Z=(y+a(z))\delx + b(z)\dely.$$
     Observe that replacing $Z$ by $Z-a(0)X-b(0)Y$ will not  affect
     the bracket relations between $Z$, $Y$ and $X$, so we
      may assume that $a(0)=b(0)=0$.
      
Given a point $p=(p_1,p_2,p_3)$ in the domain of such a coordinate chart, the vector field $U=Z-(p_2+a(p_3))X-b(p_3)Y$ is nonzero and vanishes at $p$.
At $p$,
    $$\left[ U,\delx \right]=0, \ \left[ U,\dely \right]=- \delx,\
    \left[ U,\delz \right]=-a'(p_3)\delx-b'(p_3)\dely.$$
    Since $U$ belongs to $\heis(3)$, hence 
    is conformal for $g$,  the matrix
    $$A=\left(  \begin{array}{ccc}
                                                       0&-1&a'(p_3)\\
                                                       0&0&b'(p_3)\\
                                                       0&0&0\\
                                                      \end{array}
                                                    \right) $$
                                                    which is the
                                                    matrix of $\nabla
                                                    U(p)$ in the basis
                                                    $\{ \delx, \dely, \delz \}$,  must satisfy the identity
                                                    $$g_p(A \cdot ,
                                                    \cdot )+g_p(\cdot
                                                    ,A \cdot )=\alpha g_p,  \ \alpha \in \RR$$

                                                    The matrix $A$ generates a $1$-parameter group
                                                    $\{e^{tA} \}$ in  $\RR \times \OO(1,2)$, which is nontrivial because 
                                                     the rank of $A$ is at least $1$. If the rank of $A$ were $1$, 
                                                      then $\{e^{tA} \}$ would fix 
                                                    two linearly independent vectors. But no nontrivial flow in  $\RR \times \OO(1,2)$
                                                     has this property, so that we infer $b'(p_3) \not =0$. 
  
  As $p$ was arbitrary, we may assume the
  derivative $b'$ does not vanish at any point of such a coordinate chart.  
  Now let $\psi$ be a smooth  diffeomorphism on an interval around $0$ such that $\psi(0)=0$ and $b(\psi(z))=z$. 
  The transformation
   $$ \varphi: (x,y,z) \mapsto (x,y -a(\psi(z)),\psi(z)).$$
    then yields  a local diffeomorphism fixing the origin.  Applying
    $(\varphi^{-1})_*$ to the generators yields
   $$ Z'=y \delx+z \dely, \ Y'=\dely, \ {\rm and} \ X'=\delx$$
   which are conformal for the metric $g'=\varphi^*g$.
   
 Again, let $p=(p_1,p_2,p_3)$ be a point in our coordinate chart.  The vector field $U'=Z'-p_2 X'-p_3Y'$ vanishes at 
    $p$ and is a conformal for  $g'$. A straigthforward computation yields
     $$ \left[ U',\delx \right]=0, \ \left[ U',\dely \right]=-\delx, \
     {\rm and}\  \left[ U',\delz \right]=-\dely$$
     everywhere. 
     Now the matrix
     $$B=\left(  \begin{array}{ccc}
                                                       0&1&0\\
                                                       0&0&1\\
                                                       0&0&0\\
                                                      \end{array}
                                                    \right) $$
                                                    satisfies 
 $$g_p'(B \cdot , \cdot ) + g_p'(\cdot ,B \cdot)=\alpha g_p', \
 \alpha \in \RR$$ 
The matrix of 
 $g_p'$ in the basis
 $ \left\{  \delx,\dely,\delz \right\}$ is thus of the form 
 $$\left(  \begin{array}{ccc}
                                                       0&0&-\beta(p)\\
                                                       0&\beta(p)&0\\
                                                       -\beta(p)&0&\gamma(p)\\
                                                      \end{array}
                                                    \right), \ \beta(p) >0$$
                                 Replace $g'$ by $\frac{1}{\beta} g'$, which amounts to assuming $\beta=1$.

 Now, $X'$ and $Y'$ are conformal Killing fields for $g'$. But $g'(Y',Y')=1$, and $Y'$ commutes with $X'$ and $Y'$.
  It follows that $X'$ and $Y'$ are actually isometric Killing fields for $g'$. In particular, the 
    function  $\gamma$ only depends on the variable 
 $z$, and  the metric $g'$ is:
 $$ -2 dx dz+ dy^2+\gamma(z)dz^2.$$
  Now, if $\zeta(z)$ is an antiderivative of $\gamma(z)/2$, then the change of coordinates
  $$ (x,y,z) \mapsto (x+\zeta(z),y,z)$$ converts  $g'$ to
  $-2dx dz+dt^2$. We conclude that the conformal class $[g]$,
  restricted to a nonempty open subset of $\Omega$, contains a flat metric.
\end{preuve}

\subsection{Case $\liezx$ is isomorphic to $\mathfrak{aff}(\RR) \oplus \RR$}

Let $\Sigma$ be as in proposition \ref{prop.propertiesdim3} (3) and $x_0 \in \Sigma$.  Let $Y$ and $Z$ be further generators of
$\liezx$ on a neighborhood of $x_0$ such that
 $[Y,Z]=Z$ and all other brackets are zero. The isotropy at
$x_0$ can be of three types:
\begin{itemize}
  \item isotropy generated by $U \in \liez_X$ transverse to
$\operatorname{Span}(Z,X)$:  The tangent vector $X(x_0)$ is lightlike and
fixed by the isotropy.  
Rescaling $U$ if necessary gives
$$D_{x_0}\varphi_U^t = 
\left( \begin{array}{ccc}
1 & 0 & 0\\
0 & e^{t} & 0\\
0 & 0 & e^{2t}\\
\end{array}
\right).
$$
By theorem \ref{thm.fm.linearization} and proposition
\ref{prop.stable.flat}, $(M,g)$ is conformally flat, a contradiction.

\item isotropy generated by $Z$: Then $D_{x_0}\varphi_Z^t$ is trivial,
  which implies conformal flatness by theorem \ref{thm.fm.linearization}. We thus discard this case too.
\item isotropy generated by $Z+ cX$ with $c \not = 0$: We handle this case below.
\end{itemize}

First we construct a model for the geometry on $\Sigma$.
Consider $\RR^2$ with coordinates $(x,y)$, and denote by
${{\mathcal H}}^+$ the upper half-space defined by  $y>0$.
On ${{\mathcal H}}^+$, let $A=y
\frac{\partial}{\partial y}$, $B=y \frac{\partial}{\partial x}$, and
$C= \frac{\partial}{\partial x}$. The only nontrivial bracket relation
is $[A,B]=B$, so that the Lie algebra $\lieh$  generated by $A,B,$ and
$C$
  is isomorphic to ${\mathfrak{aff}}(\RR) \oplus \RR$. These vector
fields are complete and integrate to a genuine action of
  $H \simeq {\operatorname{Aff}(\RR) \times \RR}$ on $\calh^+$ given by
the affine transformations:
  $$ \left(  \begin{array}{cc} 1 & b\\ 0 & e^a \end{array} \right) + 
\left(  \begin{array}{c} c \\ 0 \end{array} \right), \ a,b,c \in \RR.$$

  On ${{\mathcal H}}^+$, the isotropy for the local action of $\lieh$ is
always generated by
  an element of
   $\operatorname{Span}(B,C)$ transverse to $\RR C$. Thus the
local action of $\liezx$ on $\Sigma$ is locally modeled
   on that of $\lieh$
    on ${{\mathcal H}}^+$.

Let $G$ be the $4$-dimensional Lie group given by the
following transformations of ${{\mathcal H}}^+$:
$$ (x,y) \mapsto (x + \beta y + \gamma \ln(y) + \tau, e^{\alpha} y)
\qquad \alpha, \beta, \gamma, \tau \in \RR$$

In coordinates $(\alpha, \beta, \gamma, \tau)$, the
  subgroup comprising elements of the form $(0,\beta, \gamma, \tau)$ is normal and abelian. The action of
$(\alpha,0,0,0)$ on it is given by the matrix:
  $$ \left( \begin{array}{ccc} e^{- \alpha} & 0 & 0 \\
  0 & 1 &0\\
  0 & - \alpha & 1\\
  \end{array} \right).$$
  The group $G$ is thus isomorphic to a semi-direct product $\RR \ltimes
\RR^3$.

Observe that $H$ is a subgroup of $G$, corresponding to $\gamma=0$.
We have the following Liouville Theorem for $(G,{{\mathcal H}}^+)$:
\begin{lemme}
  \label{lem.liouville}
  Let $U$ and $V$ be two connected open subsets of ${{\mathcal H}}^+$,
and $f: U \to V$ a diffeomorphism such that $f_*C=C$ and
   $f_*(\lieh)=\lieh$. Then $f$ is the restriction of a unique element
of $G$.
\end{lemme}

\begin{preuve}
  Because $f$ preserves $C$, it is of the form $f(x,y)=(x + \eta(y),
\psi(y)).$
  Now $f$ also preserves $\lieh$, hence $f_*$ acts as an automorphism of
$\lieh$. In particular $f_*B=bB$, with $b \not = 0$.
   We get $\psi(y)=\frac{1}{b}y$, showing that $b>0$.  Set $b=e^{\alpha}$.
   Next,
$f_*(A)$ has the form $A+cB+dC$, from which we deduce
     $$ \eta(y)= \beta y + \gamma \ln(y) + \tau, \qquad \beta, \gamma,
     \tau \in \RR$$
\end{preuve}

\begin{corollaire}
  The surface $\Sigma$ is endowed with a $(G,{{\mathcal H}}^+)$-structure.
\end{corollaire}

\begin{preuve}
  Given an open subset $U \subset \Sigma$, denote $\liezx(U)$ the Lie
  algebra of all local conformal vector fields defined on $U$ commuting
  with $X$. 
  For each $x_0 \in \Sigma$, there exists a small neighborhood $U$
containing $x_0$, an open subset $V \subset {{\mathcal H}}^+$, and a
diffeomorphism
   $\psi: U \to V$ such that $\psi_*(\liezx(U))=\left. \lieh \right|_{V}$ and
$\psi_*(X)=C$. The corollary then
   follows from Lemma \ref{lem.liouville}.
\end{preuve}

  \begin{lemme}
   \label{lem.completeness}
   The $(G,{{\mathcal H}}^+)$-structure on $\Sigma$ is complete.
  \end{lemme}

  \begin{preuve}
   Let $\tilde{\Sigma}$ be the universal cover of $\Sigma$ and
    $\delta: \tilde{\Sigma} \to {{\mathcal H}}^+$ a developing map,
    with associated holonomy morphism
$\rho: \pi_1(\Sigma) \to G$.  By construction,
$\delta_*(\tilde{X})=C$, where
     $\tilde{X}$ is the lift of $X$ to $\tilde{\Sigma}$.
     In particular, the relation $\delta \circ
{\varphi}_{\tilde{X}}^t=\varphi_C^t \circ \delta$ shows that
$\delta(\tilde{\Sigma})$ is a union of lines $y \equiv c$, for $c \in \RR$. Thus
     $\delta(\tilde{\Sigma})$ is a horizontal  strip in ${\mathcal
       H}^+$.

     Next observe that $G$ preserves the degenerate metric
      $h_0= dy^2/y^2$ on ${{\mathcal H}}^+$, from which $\Sigma$
inherits a degenerate metric $h$.  On a $2$-fold cover of
      $\Sigma$, there is $W \in \mathcal{X}(\Sigma)$
transverse to $X$ and satisfying $h(W,W)=1$.
 For $\tilde{x} \in \tilde{\Sigma}$, the trajectory
$\delta(\varphi_{\tilde{W}}^t.{\tilde{x}})$ is a curve in ${{\mathcal
    H}}^+$ with velocity of constant $h_0$-length $1$,
      defined on $\RR$ since $W$ is complete on 
$\Sigma$.
       This curve
must cross every horizontal line in ${{\mathcal H}}^+$; we conclude
that $\delta(\tilde{\Sigma})={{\mathcal H}}^+$.  In fact, the open set
$\Omega=\{ {\varphi}_{\tilde{X}}^t \varphi_{\tilde{W}}^s.\tilde{x} \ | \ s,t \in
\RR \} \subset \tilde{\Sigma}$ is mapped diffeomorphically by
$\delta$ onto ${{\mathcal H}}^+$.  The boundary
         $\partial \Omega$ is empty. Indeed, if $x
\in \partial \Omega$, then there is $x' \in \Omega$
         satisfying $\delta(x')=\delta(x) =y$.  Disjoint neighborhoods of both $x$
         and $x'$ in $\tilde{\Sigma}$ map diffeomorphically under $\delta$ to a
         neighborhood of $y$.  On the other hand, both neighborhoods
         in $\tilde{\Sigma}$ intersect $\Omega$, which contradicts injectivity of $\delta$ on $\Omega$.
   We conclude that $\Omega= \tilde{\Sigma}$, and completeness follows.
        \end{preuve}

The developing map $\delta$ identifies $\pi_1(\Sigma)$ with a discrete subgroup of $G$, and by proposition \ref{prop.propertiesdim3} (3),
 this discrete group contains an index $2$ subgroup $\Lambda$ isomorphic to ${\bf Z}^2$. There must be in $\Lambda$
  an element $\gamma_0=(\alpha_0, \beta_0, \gamma_0, \tau_0)$ with $\alpha_0 \not = 0$, otherwise $\Lambda$ would preserve the lines 
  $y\equiv c$ in ${\mathcal H}^+$, and could not act cocompactly. Now it is 
  readily checked that the centralizer $L$ of $\gamma_0$ in $G$
   comprises elements of $G$ of the form 
   $\left( \alpha,
     \frac{\beta_0}{e^{\alpha_0}-1}(e^{\alpha}-1),\frac{\gamma_0}{\alpha_0}\alpha,\tau
     \right), \alpha, \tau \in \RR$.
   
 Hence $L$ is isomorphic to $\RR^2$, and $\{ \varphi_C^t \}$ acts by translations on the torus $L/ \Lambda$.
  Because $\overline{\{ \varphi_C^t \}}$ is compact in 
$L/\Lambda$, this would make $\{ \varphi_X^t \}$  relatively compact in
      $\operatorname{Diff}(\Sigma)$.  
      Then there are $t_k \rightarrow
      \infty$ and $x \in \Sigma$ such that  $\left\{ D_x (\left. \varphi^{t_k}_X
      \right|_{\Sigma}) \right\}$ and $\left\{ \left( D_x (\left. \varphi^{t_k}_X
      \right|_{\Sigma}) \right)^{-1} \right\}$ are bounded.  Since $\Sigma$ has codimension one, this implies boundedness of $D_x
      \varphi^{t_k}_X$ and $(D_x
      \varphi^{t_k}_X)^{-1}$.  Then there is a holonomy sequence for
      $\{ \varphi^{t_k}_X \}$ at $x$ contained in $P^+$, which implies conformal flatness by theorem 
       \ref{thm.translation.holonomy}. We have reached the desired contradiction.

\section{Case $\liezx$ is $2$-dimensional}
\label{sec.2dimensional}
The case in which $\liezx$ is $2$-dimensional, necessarily isomorphic
to $\RR^2$, is the most involved.  In subsection \ref{sec.global-action} below,
we show that $\liezx$ generates a global action of $S^1 \times \RR$ on
$M$.  The remaining subsections follow the main ideas of section 6 of
\cite{mp.confdambra}---also the most difficult part of the proof in that paper---to arrive at a contradiction.
The contradiction is with the standing assumptions that $\{
\varphi_X^t \}$ is essential and $(M,g)$ is not conformally flat.
From these assumptions we can immediately deduce the following 
facts about $\liezx$-orbits:
 
\begin{lemme}
 \label{lem.orbits-dim1}
 \begin{enumerate}
  \item There is no $\liezx$-orbit of dimension $0$.
  \item There exists a closed $\liezx$-orbit $\Sigma$ of dimension $1$
 \end{enumerate}

\end{lemme}
 \begin{preuve}
  Corollary \ref{coro.isotropy-dim2} rules out any $\liezx$-orbit of dimension $0$.
  Because $M$ is compact, there are recurrent points for $\{ \varphi_X^t \}$. At such points, the isotropy 
  is nontrivial by proposition \ref{prop.recurrence}, and the $\liezx$-orbit is of dimension $1$.  It is closed because of theorem \ref{thm.gromov.stratification} (2) and point (1).
 \end{preuve}

\subsection{Global conformal action on $M$}
\label{sec.global-action}

The result of this subsection is that, possibly after passing to a
finite cover of $M$, the vector fields in $\liezx$ are globally
defined, necessarily complete, generating a conformal action of a cylinder $Z_X$. 

\begin{proposition}
  \label{prop.global.R2}
After possibly replacing $M$ by a finite cover, every local conformal
vector field on $M$ commuting with $X$ extends to 
 a conformal vector field defined on all of $M$.  The resulting
 subalgebra of $\mathcal{X}^{conf}(M)$, isomorphic to $\liezx$,
 integrates to a subgroup group $Z_X \leq \mbox{Conf}(M)$ isomorphic
 to $S^1 \times \RR$, acting locally freely on 
   an open, dense subset of $M$.
\end{proposition}

\begin{preuve}
 Let $\tilde{M}$ be the universal cover of $M$, with group of deck
 transformations $\Gamma < \mbox{Conf}(\tilde{M})$.  Denote the lifts of $X$ and $\liezx$ to
  $\tilde{M}$ also by $X$ and $\liezx$. By Amores' theorem
  \cite{amores.killing} (see section \ref{sec.gromov.frobenius}), the
  lifts of
  $\liezx$ form a globally defined subalgebra of $\mathcal{X}^{conf}(\tilde{M})$.  Let $\Gamma_0$ be the kernel of the holonomy representation
  of $\Gamma$ on $\liezx$, and set $\tm'=\tm/\Gamma_0$.  We will again
  denote by $\liezx$ the corresponding subalgebra of
  $\mathcal{X}^{conf}(\tm')$.  The manifold $M$ is a quotient of $\tm'$
    by a group $\Gamma' \cong \Gamma/\Gamma_0$, and the holonomy representation of $\Gamma'$ on $\liezx$
     is faithful.
     
     Lift the closed orbit $\Sigma$
      to $\tm'$, and let $\tilde{\Sigma}$ be a connected component of this lift. It is the
  $\liezx$-orbit of a point $\tilde{x}_0$ in $\tm'$.
  
      \begin{lemme}
        \label{lemme.tilde.sigma.compact}
   The manifold $\tilde{\Sigma} \subset \tm'$ is a circle.
  \end{lemme}

  \begin{preuve}
  Assume, for the sake of contradiction, that  $\tilde{\Sigma}$ is diffeomorphic to $\RR$.  In this case, the stabilizer of 
  $\tilde{\Sigma}$ in $\Gamma'$ is a cyclic subgroup 
  $\Gamma_{\tilde{\Sigma}}' = \langle \gamma \rangle \cong {\bf Z}$.
First suppose $X(\tilde{x}_0) \neq 0$, which implies $X$ is nonvanishing on $\tilde{\Sigma}$.  Then there is $T_0 \not =0$ 
such that $\gamma \varphi^{T_0}_X \in
\mbox{Is}_X^{loc}(\tilde{x}_0)$.  This group has finitely-many
components by theorem \ref{thm.isotropy.algebraic},  so a power
$$(\gamma \varphi^{T_0}_X)^\ell =
\gamma^{ \ell} \varphi^{\ell T_0}_X = \varphi^{1}_Z$$
on a neighborhood $U$ of $\tilde{x}_0$, where $Z$
is a generator of $\mathfrak{Is}_X(\tilde{x}_0)$, and we assume
$\varphi^t_Z$ is defined on $U$ for all $t \in [0,1]$.  
Let  $Y = Z - \ell T_0 X$.  For $t \in [0,1]$ and $x \in U$, the composition
$\varphi^{-t T_0}_X \circ \varphi^t_Z(x)$ is well-defined. Because
$X$ and $Z$ commute, $\varphi^t_Y(x) = \varphi^{-t T_0}_X \circ
\varphi^t_Z(x)$, and now
$$ \left. \varphi^1_Y \right|_U = \left. \gamma^k \right|_U.$$
Then $(\gamma^k)_* Y = Y$ on $U$, which implies by analyticity that
$(\gamma^k)_*Y=Y$ on $\tm'$.  Because $\Gamma'$ fixes $X$ and $Y$ is
independent of $X$, the element $\gamma^k \neq 1$ would centralize $\liezx$,
contradicting faithfulness of $\Gamma'$ on $\liezx$.

Next suppose $X(\tilde{x}_0) = 0$, which implies $X \equiv 0$ on $\tilde{\Sigma}$, and let $Z \in \liezx \backslash \RR X$.  Such $Z$ is nonvanishing by corollary \ref{coro.isotropy-dim2}. 
There is $T_0 \not = 0$
such that $\varphi^{T_0}_Z(\tilde{x}_0) = \gamma(\tilde{x}_0)$, and 
$\varphi^t_Z$ is defined on a neighborhood $U$ of $\tilde{x}_0$ for
all $t \in [0,T_0]$.  Then $\gamma^{-1} \varphi^{T_0}_Z \in
\mbox{Is}_X^{loc}(\tilde{x}_0)$ on $U$.  Again by theorem
\ref{thm.isotropy.algebraic}, the latter group contains $\{
\varphi^s_X \}$ as a finite-index subgroup.  Thus, for some $k \in 
\BZ$ and $S_0 \in \RR$,
$$ (\gamma^{-1} \circ \varphi^{T_0}_Z)^k = \varphi^{S_0}_X$$
on a neighborhood of $\tilde{x}_0$. 

Because $[X,Z] = 0$, the differential $D_{\tilde{x}_0} \varphi^{S_0}_X
(Z(\tilde{x}_0)) = Z(\tilde{x}_0)$.  Next 
$$D_{\tilde{x}_0} (\gamma^{-1} \circ \varphi^{T_0}_Z) (Z(\tilde{x}_0))
= \pm Z(\tilde{x}_0) = (\gamma^{-1}_* Z)(\tilde{x}_0)$$
Then $\gamma_*(Z) = \pm Z + \beta X$ for some
$\beta \in \RR$.  If $\gamma_*(Z)$ were congruent to $- Z$ modulo $\RR
X$, then, given that $\varphi^{-T_0}_Z \gamma(\tilde{x}_0) =
\tilde{x}_0$, we would have
$$ \gamma^{-1}( \tilde{x}_0) = \gamma^{-1} \varphi^{-T_0}_Z
\gamma(\tilde{x}_0) = \varphi^{T_0}_Z(\tilde{x}_0) = \gamma(\tilde{x}_0)$$
which is absurd, since $\Gamma$ acts freely and $\gamma$ has infinite order.  Thus $\gamma_*(Z) = Z + \beta X$.
Now there is $S_1$ such that
$$  (\gamma^{-1} \circ \varphi^{T_0}_Z)^k = \gamma^{-k}
\varphi^{S_1}_X \varphi^{T_0}_Z
 \in \mbox{Is}^{loc}_X(\tilde{x}_0)$$
Thus, for some $S_2$,
$$ \gamma^k = \varphi^{S_2}_X  \varphi^{ T_0}_Z$$
on $U$.  Let $Y = S_2 X + T_0 Z$.  As
above, the flow along $Y$ is well-defined on $U$ for $t \in [0,1]$, and 
$$ \left. \varphi^1_Y \right|_U = \left. \gamma^k \right|_U.$$
 This leads again to $(\gamma^k)_*Y=Y$ on $\tm'$, yielding a contradiction with the faithfulness of the action of $\Gamma'$ on $\liezx$.
 \end{preuve}
 
 Next we show that $\liezx$ globalizes on a neighborhood of the closed
 orbit $\tilde{\Sigma}$.

  \begin{lemme}
    \label{lem.Y.periodic}
    There exists $Y \in \liezx$, and $\tilde{N}$ an open neighborhood  of $\tilde{\Sigma}$, such that for all $y \in \tilde{N}$,
     the orbit $\{ \varphi_Y^t.y \}$ is defined on $\RR$, included in $\tilde{N}$, and $1$-periodic.
      \end{lemme}

      \begin{preuve}
        If $X(\tilde{x}_0) \neq 0$, then $\{ \varphi^t_X .\tilde{x}_0 \}$ is
  periodic.  Because $\mbox{Is}_X^{loc}(\tilde{x}_0)$ has finitely-many
components, there are $T_0>0$, $S_0 \in \RR$ such that $\varphi^{T_0}_X \circ
\varphi^{S_0}_Z$ is trivial on a neighborhood of $\tilde{x}_0$, where
$Z$ generates $\mathfrak{Is}_X(\tilde{x}_0)$.   Note that, for any
$S$, the flow $\{ \varphi^s_Z \}$ is defined for $s \in [0,S]$ in some
neighborhood of $\tilde{x}_0$, and it commutes with $\{ \varphi^t_X \}$.  Now $Y = T_0 X + S_0 Z \in \liezx$ generates a periodic flow $\{ \varphi^{rT_0}_X \circ \varphi^{r S_0}_Z \}$, defined for $r \in \RR$ in a neighborhood $U$ of
$\tilde{x}_0$.  We can then put  $\tilde{N}:= \bigcup_{t \in \RR} \varphi^t_Y(U)$, which is indeed an open neighborhood
 of $\tilde{\Sigma}$ satisfying the conclusions of the lemma.

If $X$ vanishes on $\tilde{\Sigma}$, then by 
corollary \ref{coro.isotropy-dim2}, any $Z \in \liezx
\backslash \RR X$ does not vanish on $\tilde{\Sigma}$.
Then $\bar{Z} = \left. Z \right|_{\tilde{\Sigma}}$ is nonvanishing and
complete because $\tilde{\Sigma}$ is compact by lemma \ref{lemme.tilde.sigma.compact}.  Let $T_0 >0$ be such that $\varphi_{\bar{Z}}^{T_0}.\tilde{x_0} = \tilde{x}_0$.  The flow along $Z$ is defined on $[0,T_0]$ in a
neighborhood $\tilde{N}$ of $\tilde{\Sigma}$.  By theorem
\ref{thm.isotropy.algebraic}, $\{ \varphi^s_X \}$ has finite index in
$\mbox{Is}_X^{loc}(\tilde{x}_0)$ and is
noncompact because it is an essential flow.   After replacing $T_0$ by a finite
integer multiple, and shrinking $\tilde{N}$ if necessary, 
$\varphi_{Z}^{T_0} \varphi^{S_0}_X$ will be trivial in
restriction to $\tilde{N}$, for some $S_0 \in \RR$.  Because $X$ is complete
and $[Z,X]=0$, the flow along $Y = T_0 Z + S_0 X$ restricted to $\tilde{N}$ is periodic with
period $1$ and complete. Again,  $\tilde{N}:= \bigcup_{t \in \RR} \varphi^t_Y(U)$ 
 is  an open neighborhood
 of $\tilde{\Sigma}$ satisfying the conclusions of the lemma.
\end{preuve}

Let $\Omega = \bigcup_{t \in \RR}  \varphi^t_X .\tilde{N}$, an open set; let
$Y$ be as in lemma \ref{lem.Y.periodic}.  For every $y \in \Omega$,  the  pseudo-orbit $\{ \varphi_Y^t.y \}$
 is defined for every $t \in \RR$, included in $\Omega$, and
 $1$-periodic.  Any $Z \in \liezx$ equals $r_0 X + s_0 Y$, for some
 $r_0$ and $s_0$, so for every $y \in \tm'$, we have $\varphi_Z^t.y=\varphi_X^{tr_0} \circ \varphi_Y^{ts_0}.y$, for all $t$ such that the
  expression is defined. For $y \in \Omega$, in particular,
  $\varphi_Z^t.y$ is defined for every $t \in \RR$ and lies in $\Omega$.
  All vector fields of $\left. \liezx\right|_{\Omega}$ are thus complete, defining an action of a $2$-dimensional  abelian Lie group
   $Z_X$
    on $\Omega$.  For the time being, we consider $Z_X$ as a subgroup of $\Conf(\Omega)$.
    Since the flow $\{ \varphi_Y^t \}$ is cyclic on $\Omega$, we infer that $Z_X$ is isomorphic  to a cylinder $S^1 \times \RR$
     or a torus ${\bf T}^2$.
     The latter possibility would mean that $\overline {\{ \varphi_X^t \} }$ is compact in $\Conf(\Omega)$, 
      hence in $\Conf(\tm')$, and finally in $\Conf(M)$,
      contradicting the essentiality hypothesis. 
      
      Now assume $Z_X$ is a cylinder.  There exists
      a homomorphism $\rho: \Gamma' \to \Aut(Z_X)$ integrating the representation $\Gamma' \to \Aut(\liezx)$.
      Indeed, for $\gamma \in \Gamma'$ and $Z \in \liezx$
      the local flows 
      $\{ \gamma \varphi_Z^{s} \gamma^{-1} \}$ and $
      \{\varphi_{\gamma_* Z}^{s} \}$ coincide on $\Omega$; in particular, 
      $ \gamma \varphi_Z^{1} \gamma^{-1}$ is well-defined on $\Omega$ and belongs to $Z_X$.
          
  Because $\rho(\Gamma')$ fixes $\varphi^t_X$ for all $t$, it follows
  that $|\rho(\Gamma')| \leq 2$.  Because $\rho$ is faithful,
  $|\Gamma'| \leq 2$.
    Then $\tm'$ is compact, and all vector fields of $\liezx$ are
    complete.  Thus $\liezx$ 
   integrates to the action of a group $Z_X \cong S^1 \times \RR$, as claimed in the proposition.
    
    We replace $M$ by $\tm'$ in the sequel.
   To complete the proof of the proposition, it remains to check that the action of $Z_X$ is locally free on a dense open subset
    of $M$. If not, there would be a nonempty open subset $U$ on
   which  all $Z_X$-orbits have dimension $1$.  The identity component of the isotropy group 
$\mbox{Is}_X^{loc}(x_0)$ at a point $x_0 \in U$ would fix a nonzero vector $v$
tangent to the orbit and act trivially on the quotient $T_{x_0}
M/\RR v$.  Basic linear algebra shows that the differential of the
isotropy at $x_0$, identified with a subgroup of $\mbox{CO}(1,2)$,
must be trivial in this case.  On the other hand, theorem
\ref{thm.fm.linearization} says, given our assumption on $(M,g)$, that
the isotropy must be linearizable, yielding a contradiction.
    \end{preuve}

Now we replace $M$ by the finite covering given by the above proposition.   
The  remainder of this section follows the arguments of section 6 of
  \cite{mp.confdambra}, with some simplifications available in
  dimension three.  We obtain directly from the
  $Z_X$-action a foliation by degenerate surfaces on an open, dense
  subset $\Omega \subset M$ in section \ref{sec.description-orbits}.  The leaves are projections of
  leaves of an integrable
  distribution in a reduction of the Cartan bundle over $\Omega$ (section \ref{sec.geometric-2dim}).
  Then we appeal to \cite{mp.confdambra} to better understand how
  $2$-dimensional orbits can accumulate on closed orbits
  (section \ref{sec.accumulation}).  A contradiction
  ultimately results from the fact that only finitely  many
  $1$-dimensional orbits can attract 2-dimensional orbits, while the
  $2$-dimensional orbits should accumulate on uncountably many
  distinct $1$-dimensional orbits
  (section \ref{sec.stratification-conclusion}).

\subsection{Description of the $Z_X$-orbits}
\label{sec.description-orbits}

\subsubsection{All orbits are degenerate}
  
\begin{proposition}
\label{prop.R2.orbits}
  There are two types of $Z_X$-orbits in $M$, both of which occur:
  \begin{enumerate}
    \item Circular lightlike orbits, with linear, unipotent isotropy.
      These form an analytic subset $\Sigma$ of $M$.
  \item Cylindrical orbits on which the metric is degenerate.  Each of
    these contains an orbit of type (1) in its closure.  These fill an
    open, dense subset $\Omega_f$ of $M$.
    \end{enumerate}
  \end{proposition}

  \begin{preuve}
    By corollary \ref{cor.small-stabilizer}, there are no $Z_X$-fixed
points; moreover, all 1-dimensional orbits are closed---otherwise,
there would be a fixed point in the closure by theorem \ref{thm.gromov.stratification} (2).
The closed, lightlike orbits are precisely the zero-set of the
analytic function, given by a choice of metric $g$ in the conformal
class and $Y \in \liezx \backslash \RR X$ by
$$ \varphi(x) = g_x(X,X)^2 + g_x(Y,Y)^2 + g_x(X,Y)^2$$

Since we assume that $(M,g)$ is not conformally flat, the elements of the isotropy algebra 
 can be assumed linearizable by theorem
\ref{thm.fm.linearization}.  In a 1-dimensional, lightlike orbit,
the $Z_X$-isotropy fixes a lightlike tangent vector.
Such isotropy easily seen to be balanced or unipotent.  In the first
case, the isotropy is stable
and leads to conformal flatness by
proposition \ref{prop.stable.flat}.  We have
proved all the claimed properties of orbits of type (1).

Let $\Omega_f$ be the set on which $Z_X$ acts locally freely; it is
open and dense by proposition \ref{prop.global.R2}.  If there were a
closed, 2-dimensional orbit, then $\{ \varphi^t_X \}$ would have a
recurrent point on this orbit,
contradicting proposition \ref{prop.recurrence}.
Thus every 2-dimensional orbit is not closed, and, by theorem \ref{thm.gromov.stratification} (2), contains a closed
1-dimensional orbit in its closure.

Now we focus on the linear part of a holonomy sequence $\{ p_k \}$ for
an unbounded sequence $\{ h_k \} \subset Z_X$ at
$x \in \Omega_f$.
By proposition \ref{prop:as_for_inv_sections}, the subspace $\liezx(x)
\subset T_x M$
is approximately stable for $\{ D_x h_k \}$, because
$Z_X$ centralizes $\liezx$.  This means that $\omega_{\hat{x}}(\liezx)$ belongs
modulo $\liep$ to
$(\lieg/\liep)^{AS}(p_k)$, for any $\hat{x} \in \pi^{-1}(x)$.  Assuming $p_k$ is in $A'P^+$-form, the
presence of a 2-dimensional approximately stable subspace in
$\lieg/\liep$ makes it contracting, balanced, mixed, or of bounded
distortion, as in definition \ref{def:holonomy_trichotomy}.

Let $\Omega \subseteq \Omega_f$ be the open subset where the Cotton
tensor is nonzero; by our standing assumption, it is also dense.  For $x \in \Omega$, a holonomy sequence $\{ p_k \}$
as above is of
bounded distortion or mixed type by proposition
\ref{prop.balanced.flat}.  For these types, $(\lieg/\liep)^{AS}(p_k)$ is a
degenerate plane.  Thus all orbits in $\Omega$ are degenerate.  On the
other hand, if a point $x \in \Omega_f$ has Riemannian or Lorentzian orbit,
then so do the points in a neighborhood of $x$.  Since $\Omega$ is
open and dense, we conclude that all orbits of $\Omega_f$---that is,
all 2-dimensional orbits---are degenerate.


We rule out nondegenerate orbits of dimension 1 in lemma
\ref{lem:1d.degenerate} directly below, which completes the proof.
 \end{preuve}

\begin{lemme}
  \label{lem:1d.degenerate}
There are no 1-dimensional spacelike or timelike orbits.
\end{lemme}

The proof of this lemma makes use of the exponential map of the Cartan
connection, defined in subsection \ref{subsec.stability.propagation}.
The exponential map will appear frequently in the remainder of this
section.

\begin{proof}
If a point $x$ has a timelike orbit, then $\varphi$ is negative in a
neighborhood of $x$.  This neighborhood intersects $\Omega_f$, where
$\varphi$ is nonnegative, a contradiction.

Now suppose that $x$ has a 1-dimensional spacelike orbit, and let
$\hat{x} \in \pi^{-1}(x)$ be such that the isotropy image
$(\hat{I}_X)_{\hat{x}}$ is linear, contained in $G_0 = \mbox{CO}(1,2)$.
  It fixes a spacelike vector, which makes it conjugate in $G_0$ to a
  1-parameter diagonal group
  $$\{ \hat{h}^t : (x^1,x^2,x^3) \mapsto (e^tx^1,x^2,e^{-t}x^3) \}$$ 
Let $\hat{Z} \in \mathfrak{co}(1,2)$ be the generator of $\{ \hat{h}^t
\}$, corresponding under $\iota_{\hat{x}}$ to a generator $Z \in \mathfrak{Is}_X(x)$.
  Note that the corresponding linear vector field $\hat{Z}$ on
  $\RR^{1,2}$ is timelike along the line $\RR(E_1+ E_3)$ (with the metric
  $\mathbb{I}$ of section \ref{sec.cotton.tensor}).

  Let $\hat{\gamma}(t) = \exp_{\hat{x}}(t(E_1+E_3))$ and $\gamma = \pi
  \circ \hat{\gamma}$.  The lift $\hat{\gamma}$ determines a metric
  along $\gamma$ in the conformal class $[g]$ by
$$ \langle u,v \rangle_{\gamma(t)} = \mathbb{I}\left(
  \omega^{(-1)}_{\hat{\gamma}(t)} \hat{u}, \omega^{(-1)}_{\hat{\gamma}(t)}
  \hat{v} \right)$$
where $\omega^{(-1)}$ is the component of $\omega$ on $\lieg_{-1}
\cong \RR^{1,2}$, and $\hat{u}, \hat{v}$ are any lifts of $u,v$ to
$T_{\hat{\gamma}(t)} \hat{M}$.
We will approximate
$\langle Z,Z \rangle_{\gamma(t)}$
by computing
$$ \left. \frac{d}{dt} \right|_0 \omega^{(-1)}_{\hat{\gamma}(t)}(Z)
=    (\hat{E}_1+\hat{E}_3)_{\hat{x}}.\omega^{(-1)}(Z) $$
where $\hat{E}_1 + \hat{E}_3$ is the vector field on $\hm$ satisfying $\omega(\hat{E}_1 + \hat{E}_3) \equiv E_1+E_3$.

Using that $L_Z\omega(\hat{E}_1+\hat{E}_3) \equiv 0$, the Cartan curvature gives
  $$ 0 = \Omega_{\hat{x}}(\hat{E}_1+\hat{E}_3,Z) =
  (\hat{E}_1+\hat{E}_3)_{\hat{x}}.\omega(Z)+ \left[ E_1 + E_3,\hat{Z}
  \right]$$
Thus
  $$ (\hat{E}_1+\hat{E}_3)_{\hat{x}}.\omega^{(-1)}(Z) = \left[
    \hat{Z},E_1+E_3 \right]^{(-1)} = E_1-E_3$$

  Now
  $$\omega^{(-1)}_{\hat{\gamma}(t)}(Z) = t(E_1 - E_3+ R(t)), \qquad
  \lim_{t \rightarrow 0} R(t) = 0$$
  and
\begin{eqnarray*}
  \langle Z,Z \rangle_{\gamma(t)} & = & \mathbb{I}\left( t(E_1 - E_3+ R(t)),
    t(E_1 - E_3+ R(t)) \right) \\
   & = & t^2(-2 + Q(t)) \qquad \lim_{t
    \rightarrow 0} Q(t) = 0
\end{eqnarray*}
For sufficiently small $t$, this inner product is negative along
$\gamma$.  This would mean that $Z \in \liezx$ is timelike on a
nonempty open set, contradicting that all $Z_X$-orbits in the open,
dense set
$\Omega_f$ are degenerate.
    \end{proof}



    \begin{corollaire}
      \label{cor.mixed.holonomy}
      Let $\{ h_k \}  \subset Z_X$
      be an unbounded
sequence, and suppose that $h_k.x \rightarrow y$ for $x \in \Omega$
and $y \in \Sigma$.
Then
any holonomy sequence for $\{ h_k \}$ at $x$ is of
mixed type.
\end{corollaire}

\begin{proof}
We established during the proof of
proposition \ref{prop.R2.orbits} above that holonomy sequences at
points of $\Omega$ are
of bounded distortion or mixed type.

By proposition \ref{prop.R2.orbits}, there is $Z \in \liezx$ with $Z(x)$ spacelike.  It is also
approximately stable for $\{ D_x h_k \}$.  If $\{ h_k \}$ has bounded
distortion at $x$, then $\lim_k D_x h_k(Z)$ is a nonzero spacelike
vector tangent to the orbit of $y$.  This would contradict the result
from proposition \ref{prop.R2.orbits} above that the orbit of $y$ is
one-dimensional and lightlike.
\end{proof}

\subsection{$Z_X$-orbits in the Cartan bundle}

In this section we construct a $Z_X$-invariant reduction of 
$\left. \hat{M} \right|_\Omega$ and show that $Z_X$-orbits there are
tangent to a special distribution defined by the Cartan connection.
These properties will
be key to controlling the accumulation of 2-dimensional orbits in
$\Omega$ on
1-dimensional orbits in $\partial \Omega$.

\label{sec.geometric-2dim}
\subsubsection{Adjoint approximately stable spaces}

 As in \cite{mp.confdambra}, we will call $\{ p_k \}$ an \emph{ACL holonomy
sequence} at $x \in \Omega_f$ if it is in $A'P^+$-form;
it corresponds to $h_k.x \rightarrow y$ for
$y \in \Sigma$; and
the isotropy in $Z_X$ with
respect to $\hat{y} = \lim h_k.\hat{x}_k.p_k^{-1}$ is linear.
Every ACL holonomy
sequence at $x \in \Omega$ has $( \lieg/\liep)^{AS}(p_k)$  $= E_1^\perp$,
because its linear component is of mixed type by corollary \ref{cor.mixed.holonomy}.

We will next describe $\lieg^{AS}(p_k)$, which
 also reflects the nonlinear part of $\{ p_k \}$.  We
 use the $G_0$-invariant decomposition $\lieg = \lieg_{-1} \oplus
 \lieg_0 \oplus \lieg_1$ with $\lieg_{-1} \cong \RR^{1,2}$ to identify
 $E_1^\perp$ with a subspace of $\lieg$.

 \begin{proposition}[\cite{mp.confdambra} Prop 6.5]
   \label{prop.E1perp.stable}
   For $\{ p_k \}$ an ACL holonomy sequence at $x \in \Omega$,
$$    E_1^\perp \subset \lieg^{AS}(p_k)$$
\end{proposition}

\begin{preuve}
Write
  $\{p_k = d_k
\tau_k \}$.  By \cite[Lem 6.4]{mp.confdambra},
$\{ \xi_k = \ln \tau_k \}$ is contained in the line
$\RR E_1^t \mathbb{I} \subset \liep^+ \cong \RR^{1,2*}$.  Moreover,
for $D_k = \ln d_k$, the sequence $\{ e^{\beta(D_k)} \xi_k \}$ is bounded.
Now, it follows that for $\{ p_k
\}$ as above and $v \in E_1^\perp$,
$$ \Ad(p_k).v = \Ad(d_k)(v + [\xi_k,v]) = \Ad(d_k).v +
e^{\beta(D_k)}[\xi_k,v],$$
noting that $[\xi_k,v]$ is in the root space
$\lieg_{\beta}$ and $[ \xi_k, [\xi_k,v]] =
 0$.  This expression is bounded, so the desired inclusion follows.
  \end{preuve}

  \bigskip

 \subsubsection{Reduction of $\hat{M}$ over $\Omega$}

Write $\hat{\Omega}_f = \left. \hat{M} \right|_{\Omega_f}$.  Recall
that $Z_X$ acts locally freely in $\Omega_f$ with degenerate orbits,
by proposition \ref{prop.R2.orbits}.  Given $x \in \Omega_f$, the
orthogonal $n(x)$ to the orbit of $x$ is a 
lightlike line, tangent to the orbit.  We define a reduction of
$\hat{\Omega}_f$ given by the conformal frames in which this
orthogonal is the line $[E_1]$, as follows.
Denote $\mathcal{N}$ the null cone of the Lorentzian inner
product $\mathbb{I}$ in $\RR^{1,2}$, and by ${\bf P}(\mathcal{N})$ its
projectivization.  Let
\begin{eqnarray*}
\eta  :  \hat{\Omega}_f  & \rightarrow & {\bf P}(\mathcal{N}) \\
\hat{x} & \mapsto & \omega^{(-1)}_{\hat{x}}(\widehat{n(x)})
\end{eqnarray*}
where $\widehat{n(x)}$ is any lift of $n(x)$ to $T_{\hat{x}} \hat{M}$,
and $\omega^{(-1)}$ denotes the component on $\lieg_{-1} \cong \RR^{1,2}$.
This map is well-defined, and in fact analytic, in $\Omega_f$.  The
level set of $[E_1]$ is a reduction $\mathcal{R}' \subset \hat{\Omega}_f$ to $Q_0
\ltimes P^+$, where $Q_0 < G_0$ is the stabilizer of $[E_1]$.  The
$Z_X$-action preserves orbits, so it preserves the orthogonals, and it
leaves $\omega$ invariant; thus $Z_X$ preserves $\mathcal{R}'$.


Now we restrict to $\hat{\Omega} = \left. \hat{M} \right|_{\Omega}$.  Let $\lieq_1$ be the annihilator of $E_1$ in $\liep^+ \cong \RR^{1,2*}$,
with corresponding connected subgroup $Q_1 < P^+$.  Let $Q = Q_0
\ltimes Q_1< P$.  Define
$$ \mathcal{R} = \{ \hat{x}  \in \hat{\Omega} : 
\omega_{\hat{x}}^{-1}(E_1^\perp) \subset T_{\hat{x}} \mathcal{R}' \}$$

This construction and the following proposition are very similar to
\cite[Sec 6.2]{mp.confdambra}.

\begin{proposition}
  \label{prop.reduction}
The set $\mathcal{R}$ is a $Z_X$-invariant reduction of
$ \hat{\Omega}$ to $Q$.
  \end{proposition}

\begin{preuve}
Let $x \in \Omega$.  By proposition \ref{prop.R2.orbits}, there is
$h_k \rightarrow \infty$ in $Z_X$ such that $h_k.x \rightarrow y \in
\Sigma$, and the isotropy at $y$ is linear and unipotent.
There is thus an ACL
holonomy sequence $\{ p_k \}$ for $\{ h_k \}$ with respect to $\hat{x}
\in \pi^{-1}(x)$.  By corollary \ref{cor.mixed.holonomy} it is of mixed
type.  As observed in the proof of proposition \ref{prop.R2.orbits},
$\omega_{\hat{x}}(\liezx)$ belongs modulo $\liep$ to
$(\lieg/\liep)^{AS}(p_k)$, which equals $E_1^\perp$.  For $x =
\pi(\hat{x})$, the projection $\liezx(x)$ is the tangent space to the
orbit of $x$.  Thus the orthogonal $n(x)$ corresponds under $\omega$
to $[E_1]$, and $\hat{x} \in \mathcal{R}'$.  Whenever there
is an ACL holonomy sequence of mixed type with respect to $\hat{x} \in
\hat{\Omega}$, then $\hat{x} \in \mathcal{R}'$.

By proposition \ref{prop.E1perp.stable}, $E_1^\perp \subset
\lieg^{AS}(p_k)$.  By proposition \ref{prop:propagation_holonomy},
for all $X \in E_1^\perp$, for $s$ sufficiently
small, $\{ p_k \}$ is also a holonomy sequence at $\hat{\gamma}(s) =
\exp_{\hat{x}}(sX)$,  and
this point is in $\hat{\Omega}$.  Thus $\hat{\gamma}(x)$ is in $\mathcal{R}'$,
for $s$ sufficiently small, which implies $\omega^{-1}_{\hat{x}}(X) \in T_{\hat{x}}
\mathcal{R}'$.  Then $\hat{x} \in
\mathcal{R}$; moreover, every $\pi$-fiber of $\hat{\Omega}$
intersects $\mathcal{R}$.

To verify that $\mathcal{R}$ is a reduction of $\mathcal{R}'$ to $Q$,
we will express it as the level set of a smooth---actually,
analytic---map on $\mathcal{R}'$.  As $\mathcal{R}'$ is a reduction of
$\hat{\Omega}$, there is, at each $\hat{x} \in \mathcal{R}'$, a
two-dimensional subspace of $\omega(T_{\hat{x}} \mathcal{R}')$
projecting modulo $\liep$ to $E_1^\perp$, varying smoothly with $\hat{x}$.  This subspace can be
expressed as the graph of a linear homomorphism $E_1^\perp \rightarrow
\liep$, unique up to addition of a homomorphism $E_1^\perp \rightarrow
\lieq_0 \ltimes \liep^+$, corresponding to addition of vertical
vectors tangent to $\mathcal{R}'$.  Then we compose with the
projection to $\liep/\liep^+ \cong \lieg_0$ to define
$$\Phi : \mathcal{R}' \rightarrow \mbox{Hom}(E_1^\perp, \lieg_0/\lieq_0)$$
Observe that $\hat{x} \in \mathcal{R}'$ belongs to $\mathcal{R}$ if and
only if $\Phi(\hat{x}) = 0$. 

The map $\Phi$ is $Q_0 \ltimes P^+$-equivariant, where, for $\varphi \in
\mbox{Hom}(E_1^\perp, \lieg_0/\lieq_0)$,
$$ (g \cdot\tau)(\varphi) = (\left. \Ad g \right|_{\lieg_0/\lieq_0}) \circ \varphi
\circ (\left. \Ad g^{-1} \right|_{E_1^\perp}) + \left. \left( \Ad \tau
- \mbox{Id} \right) \right|_{E_1^\perp}$$
For
$\tau \in P^+$, the image $\mbox{Ad } \tau (E_1^\perp + \lieq_0) \equiv
E_1^\perp + \lieq_0$ mod $\liep^+$ if and only if $\tau \in Q_1$.  The
affine $Q_0 \ltimes P^+$-action on $\mbox{Hom}(E_1^\perp,
\lieg_0/\lieq_0)$ factors through $\mbox{Aff}(\RR)$.  The orbit of $0$
is 1-dimensional, with stabilizer $Q$.  

Because every $\pi$-fiber of $\hat{\Omega}$ intersects $\mathcal{R}$,
which is in turn contained in $\mathcal{R}'$, the image of the latter 
under $\Phi$ is contained in the 
orbit of $0$.  Now $\mathcal{R}$, the inverse image of $0$, is a smooth
$Q$-reduction of $\hat{M}$ over $\Omega$.  It is $Z_X$-invariant and
analytic because $\mathcal{R}'$ and $\omega$ are.
\end{preuve}

The geometric interpretation of $\mathcal{R}$ is as the conformal
normalized 2-frames at points $x \in \Omega$ in which the orbits are
totally geodesic  (infinitesimally at $x$).

The fact from the proof of proposition \ref{prop.reduction} that $\{ p_k \}$ is also
a holonomy sequence at $\hat{\gamma}(s) = \exp_{\hat{x}}(sX)$ for all
$X \in E_1^\perp$ gives that $\hat{\gamma}(s) \in \mathcal{R}$, for
all $s$ such that $\gamma(s) = \pi \circ \hat{\gamma}(s) \in \Omega$.
It follows that $\omega^{-1}_{\hat{x}}(E_1^\perp) \subset T_{\hat{x}} \mathcal{R}$.

   \bigskip
  
 \subsubsection{Foliation of $\mathcal{R}$}

Let $\hat{\mathcal{D}} = \omega^{-1}(E_1^\perp + \lieq)$, an
analytic distribution on $\hat{M}$.  The restriction to $\mathcal{R}$
is tangent to $\mathcal{R}$ because, from the previous paragraph,
$\omega^{-1}(E_1^\perp)\subset T\mathcal{R}$, and $\mathcal{R}$ is a
principal $Q$-bundle.  When $M$ is 3-dimensional, we can prove
integrability of $\hat{\mathcal{D}}$ in $\mathcal{R}$ without using the
Cartan curvature.

\begin{lemme}
 \label{lem.integrability}
 The distribution  $\left. \hat{\mathcal{D}}\right|_{{\mathcal{R}}}$ is integrable. The projection on $M$ of the leaves 
 of this distribution coincide with $Z_X$-orbits in $\Omega$.
\end{lemme}

\begin{preuve}
The key fact here is that the $Z_X$-orbits in $\mathcal{R}$
are tangent to $\hat{\mathcal{D}}$.  Indeed, as noted in the proofs of
propositions \ref{prop.R2.orbits} and \ref{prop.reduction},
$\omega_{\hat{x}}(\liezx)$ is congruent modulo $\liep$ to $E_1^\perp$
if there is a mixed ACL holonomy sequence with respect to $\hat{x}$, and more
generally, if $\hat{x} \in \mathcal{R}$.  Because $\mathcal{R}$ is
$Z_X$-invariant, $\liezx(\hat{x}) \subset T_{\hat{x}}\mathcal{R}$.
Thus $\liezx(\hat{x}) \subset \hat{\mathcal{D}}_{\hat{x}}$ for all $\hat{x}
\in \mathcal{R}$.

The $(Z_X \times Q)$-orbits in $\mathcal{R}$ are integral leaves for
$\hat{\mathcal{D}}$, projecting to the two-dimensional, degenerate
$Z_X$-orbits in $\Omega$.  
\end{preuve}

 \subsection{Accumulation of $2$-dimensional orbits on $1$-dimensional orbits}
 \label{sec.accumulation}
Before applying the results of the previous section, we focus on the
geometry around
$1$-dimensional $Z_X$-orbits. 
 
      \subsubsection{Plaques at $1$-dimensional orbits}

     We define a distinguished degenerate
 surface around each $1$-dimensional orbit. 
 The following proposition is an aggregate of propositions 6.11 and 6.12 (see also remark 6.14) of 
 \cite{mp.confdambra}.
 
 \begin{proposition}
  \label{prop.unipotent.extension}
    Let $y \in M$ have closed, isotropic $Z_X$-orbit and suppose that
    $\mbox{Is}_X(y)$ is linear and unipotent
    with respect to $\hat{y} \in \pi^{-1}(y)$.  Then:
    \begin{enumerate}
     \item The point $\hat{y}$ belongs to the closure $\overline{\mathcal{R}}$.
     \item Let $\gamma: [0,1] \to M$  a continuous path, smooth on $[0,1)$, such that $\gamma([0,1)) \subset \Omega$ and $\gamma(1)=y$.
      Then there exists a continuous lift $\hat{\gamma}: [0,1] \to \hat{M}$, smooth on $[0,1)$, such that $\hat{\gamma}([0,1)) \subset \mathcal{R}$, 
       and $\hat{\gamma}(1)=\hat{y}$.
     \item  There is a
        neighborhood $\mathcal{U}$ of $0$ in
        $E_1^\perp + \lieq$ such that
        $\exp_{\hat{y}}(\mathcal{U})$ is an integral submanifold
        of $\hat{\mathcal{D}}$. 
    \end{enumerate}
     \end{proposition}

       The third  point of the proposition follows rather easily from the first, because by analyticity, 
integrability of $\hat{\mathcal{D}}$ on $\mathcal{R}$ extends to
the closure.

\begin{definition}
\label{def.distinguished}
Let $\exp_{\hat{y}}(\mathcal{U})$ be as in proposition
\ref{prop.unipotent.extension}.  For $\mathcal{U}$ small enough, 
 $\pi(\exp_{\hat{y}}(\mathcal{U}))$ is a degenerate $2$-dimensional submanifold of $M$. It will 
 be called a \emph{plaque at $y$}, and denoted ${\mathcal P}_y$.
\end{definition}

Observe that if $\hat{y}$ and $\hat{y}'$ are two points of $\pi^{-1}(y)$ such that $\mbox{Is}_X(y)$ is linear and unipotent
    with respect to $\hat{y}$ and $\hat{y}'$, then $\hat{y}'=\hat{y}.q$, where $q \in Q$ (actually $q$ belongs to the subgroup $Q_0 \ltimes Q_1' \subset Q$, where $Q_1'$ is the $1$-dimensional
    subgroup of $Q_1$ normalized by $Q_0$). In particular $\Ad(q^{-1})(E_1^{\perp} + \lieq)=E_1^{\perp} + \lieq$.
    The relation 
    $\exp_{\hat{y}'}(\Ad(q^{-1}).{\mathcal
      U})=\exp_{\hat{y}}({\mathcal U}).q$ is a consequence of
    the second axiom for $\omega$ part (2), and implies that the projections on $M$ of  
    $\exp_{\hat{y}'}(\Ad(q^{-1}).{\mathcal U})$ and $\exp_{\hat{y}}({\mathcal U})$ are the same. Thus all  plaques at $y$ have 
    the same germ, in the sense that if ${\mathcal P}_y$ and ${\mathcal P}_y'$ are two of them, then 
    ${\mathcal P}_y \cap {\mathcal P}_y'$ is open in
     ${\mathcal P}_y$ and ${\mathcal P}_y'$.

 \begin{proposition}[see \cite{mp.confdambra} Rem 6.17]
  \label{prop.accumulation}
  Let $\Delta$ be a $1$-dimensional $Z_X$-orbit.  Let ${\mathcal P}_y$ be a  plaque at $y \in \Delta$. If 
  ${\mathcal O}$ is a 
   $2$-dimensional $Z_X$-orbit of $\Omega$ containing $y$ in its
   closure, then ${\mathcal P}_y \cap {\mathcal O}$ has nonempty interior 
   in ${\mathcal P}_y$. 
 \end{proposition}

\begin{preuve}
 By theorem \ref{thm.gromov.stratification} (1), the orbit ${\mathcal O}$ is a semi-analytic set. 
 We will use the following result, known as the ``curve selecting lemma.''
 \begin{lemme}[see \cite{lojasiewicz.semianalytique} Sec 19, Prop 2]
  Let $S$ be a semi-analytic subset of $M$, and $y \in \overline{S}$.
  Assuming $y$ is not an isolated point of $S$, there exists
   an analytic arc $\gamma: [0,1) \to S$ extending continuously to $1$, with $\gamma(1)=y$.
 \end{lemme}

 This lemma applied to $S={\mathcal O}$ provides a continuous path $\gamma: [0,1] \to M$, which is smooth on $[0,1)$,
  and satisfies  
 $\gamma([0,1)) \subset {\mathcal O} \subset \Omega$, as well as  $\gamma(1)=y$. 
  Let $\hat{y} \in \pi^{-1}(y)$ with linear, unipotent holonomy. 
Pproposition \ref{prop.unipotent.extension} (2) yields a continuous lift $\hat{\gamma}$ of $\gamma$, 
such that $\hat{\gamma}$ is smooth on $[0,1)$, $\hat{\gamma}([0,1)) \subset \mathcal{R}$, and $\hat{\gamma}(1)=\hat{y}$.

Because $\gamma([0,1)) \subset {\mathcal O}$, the image $\hat{\gamma}((0,1)) $ lies in an integral 
 leaf of $\hat{\mathcal{D}}$ (this leaf is the preimage of $\mathcal{O}$ in ${\mathcal R}$).
 Let $\mathcal{U}$ be a neighborhood of $0$ in $\lieg$ sufficiently
 small that for every $t \in [0,1]$, the set $\mathcal{U}$ is mapped
 diffeomorphically by $\exp_{\hat{\gamma}(t)}$
 onto its image.  Set 
  $\mathcal{V}=\mathcal{U} \cap (E_1^{\perp} + \lieq)$ and
  ${\mathcal L}_t=\exp_{\hat{\gamma}(t)}(\mathcal{V})$ for $t \in
  [0,1]$. A suitable choice of $\mathcal{U}$ ensures that ${\mathcal
    V}$ is connected and that $\pi({\mathcal L}_t)$ is a hypersurface for all $t$.
  
  \begin{lemme}
   \label{lem.prolongation}
   For $t$ close enough to $1$, $\mathcal{L}_t$ contains $\hat{y}$.
   \end{lemme}

 \begin{preuve}
 For $t_0$ close enough to $1$, the image $U=\exp_{\hat{\gamma}(t_0)}({\mathcal U})$ contains $\hat{y}$.
 Let $(X_1, \ldots, X_m)$ be a basis of $E_1^{\perp}+\lieq$, and
 $\hat{X}_1, \ldots ,\hat{X}_m$ the corresponding vector fields on $\hat{M}$ with
 $\omega(\hat{X}_i) \equiv X_i \ \forall i$.  Write
 $$\hat{\gamma}'(t)=\sum_{i=1}^m a_i(t)\hat{X}_i(\hat{\gamma}(t)) \qquad
 \forall t \in (t_0-\epsilon,1)$$
 where $a_i: (t_0-\epsilon, 1) \to \RR$ are some smooth functions.
 The ODE
 $$\beta'(t)=\sum_{i=1}^m a_i(t)\hat{X}_i(\hat{\gamma}(t))$$
is defined on the open set $\exp_{\hat{\gamma}(t_0)}({\mathcal U})$,
   but also on the submanifold $\mathcal{L}_{t_0}$. Uniqueness of
   solutions with the intial condition $\beta(t_0)=\hat{\gamma}(t_0)$
   implies that
    $\hat{\gamma}([t_0,1)) \subset \mathcal{L}_{t_0}$.  Then $\hat{y}=\hat{\gamma}(1) \in \mathcal{L}_{t_0}$ because
   $\mathcal{L}_{t_0}$ is closed in $\exp_{\hat{\gamma}(t_0)}({\mathcal U})$. \end{preuve}

A plaque ${\mathcal P}_y$ at $y$ is, by definition, the projection of $\exp_{\hat{y}}(\mathcal{W})$, where $\mathcal{W}$
 is a neighborhood of $0$ in $E_1^{\perp}+\lieq$. Now $\exp_{\hat{y}}(\mathcal{W})$ and $\mathcal{L}_{t_0}$ are two integral 
 leaves of $\hat{\mathcal{D}}$ having a common point $\hat{y}$. It follows that the intersection 
 $\exp_{\hat{y}}(\mathcal{W}) \cap \mathcal{L}_{t_0}$ is open in both $\exp_{\hat{y}}(\mathcal{W})$ and $\mathcal{L}_{t_0}$. 
 Projecting to $M$ gives an open subset of ${\mathcal P}_y \cap {\mathcal O}$. 
\end{preuve}

\begin{corollaire}
 \label{coro.uncountable}
 Let $\Delta$ be a $1$-dimensional orbit of $Z_X$. The set of $2$-dimensional $Z_X$-orbits in $\Omega$ containing $\Delta$ in their closure
  is at most countable. In particular, the set of $1$-dimensional $Z_X$-orbits which are accumulated by $2$-dimensional orbits is uncountable.
\end{corollaire}


\begin{preuve}
 Given $y \in \Delta$, the set of $2$-dimensional $Z_X$-orbits
 $\mathcal{O} \subset \Omega$ with $\Delta\subset \overline{\mathcal{O}}$
  coincides with the set of orbits $\mathcal{O}$ with
 $y \in \overline{\mathcal{O}}$. Denote this set of orbits $\mathcal{I}_y$.  Let
 $\mathcal{P}_y$ be a plaque containing $y$. For every $\mathcal{O}
 \in \mathcal{I}_y$, the interior of 
  $\mathcal{O} \cap \mathcal{P}_y$ is nonempty by proposition
  \ref{prop.accumulation}.  Choose $U_{\mathcal{O}}$ a connected component of this interior.
   If $\mathcal{O} \neq\mathcal{O}'$ are two distinct orbits in $\mathcal{I}_y$, then $U_{\mathcal{O}} \cap U_{\mathcal{O}'}=\emptyset$. A collection of pairwise disjoint 
   nonempty open subsets of $\mathcal{P}_y$ is at most countable, so
   ${\mathcal I}_y$ is also countable. \end{preuve}

\subsection{Stratification of the set of $1$-dimensional orbits,  conclusion}
\label{sec.stratification-conclusion}

Now we denote by $\Sigma$ the complement of $\Omega$.  By proposition
\ref{prop.R2.orbits}, $\Sigma$ is a nowhere dense, analytic subset of
$M$, comprising the points with
$1$-dimensional, closed, lightlike $Z_X$-orbits orbits.  For $k \in
{\bf N}$,  we denote by $\Sigma^{(k)}$ the set of smooth points of
dimension $k$ in $\Sigma$---namely, the points contained in some $k$-dimensional 
analytic submanifold of $M$ contained in $\Sigma$.  
Observe that $\Sigma^{(k)}= \emptyset$ for every $k \geq 3$, since $\Sigma$ is nowhere dense. 

For a semi-analytic set $\Sigma$, the \emph{dimension} of $\Sigma$
equals the maximal integer $k$, denoted $k_{max}$, such that
$\Sigma^{(k)} \neq \emptyset$.
The set of singular points of $\Sigma$, denoted $\Sigma_{sing}$, is the
complement $\Sigma \backslash \Sigma^{(k_{max})}$.  By 
\cite[Thm 4]{lojasiewicz.semianalytique} (see also \cite[Thm 7.2]{bierstone.milman} and \cite[Rem 7.3]{bierstone.milman}), $\Sigma_{sing}$
 is a closed, semi-analytic subset of $\Sigma$, of dimension less than
 $\mbox{dim }\Sigma$.

\begin{lemme}
 \label{lem.stratification}
 The complement of $\Sigma^{(2)}$ in $\Sigma$ is a 
finite union of $1$-dimensional $Z_X$-orbits, possibly empty.
\end{lemme}

\begin{preuve}
 If $\mbox{dim } \Sigma = 1$, then $\Sigma^{(2)}= \emptyset$, and $\Sigma_{sing}$ is actually empty. Indeed, 
  if $\Sigma_{sing} \not = \emptyset$, it is a closed semi-analytic
  set of dimension $0$, namely, a finite number of points. These are all
  $Z_X$-fixed points because $\Sigma_{sing}$ is $Z_X$-invariant. But corollary 
  \ref{coro.isotropy-dim2} rules out any
  $Z_X$-fixed points.  Thus $\Sigma$ is a compact, $Z_X$-invariant,
  analytic submanifold of $M$, which means it is a finite union of circular $Z_X$-orbits.
    
 If $\mbox{dim }\Sigma = 2$, then $\Sigma_{sing}$ is the complement of
 $\Sigma^{(2)}$ in $\Sigma$. If nonempty, it has dimension $1$ or $0$. 
 Let $\Sigma_{sing}'$ be the singular 
 set of $\Sigma_{sing}$. If $\Sigma_{sing}' \neq \emptyset$, then
 $\mbox{dim }\Sigma_{sing}'=0$, and, 
as above, it is a nonempty set of $Z_X$-fixed points, which leads to a contradiction.
 Therefore $\Sigma_{sing}' = \emptyset$, which means $\Sigma_{sing}$
 is a closed, $1$-dimensional, $Z_X$-invariant, analytic submanifold
 of $M$, again a finite union of circular $Z_X$-orbits.
\end{preuve}

Here is the key proposition leading to a contradiction.

\begin{proposition}[see \cite{mp.confdambra} Sec 6.4]
 \label{prop.no-accumulation}
 Let $y \in \Sigma^{(2)}$. Then there exists a distinguished plaque $\Sigma$ containing $y$ and contained in $\Sigma^{(2)}$.
\end{proposition}

\begin{preuve}
 We recall here the main points of the argument from
 \cite{mp.confdambra} that $\Sigma$ is covered by finitely-many
 plaques, with some simplifications for dimension 3.  Denote by
 $\lieu_+$ the unipotent subalgebra of $\lieg_0$ annihilating $E_1$;
 it is 1-dimensional in our case.  Let $\alpha: (- \epsilon, \epsilon)
 \rightarrow \Sigma$ be a differentiable path through $y$, which we
 can assume is contained in $\Sigma^{(2)}$.  For each $t$, there is a
 point $\hat{\alpha}(t) \in \pi^{-1}(\alpha(t))$ with
 $\omega_{\hat{\alpha}(t)}(\liezx) \subset \RR E_1$ modulo $\liep$ and
 with $\lieu_+ \subset \omega_{\hat{\alpha}(t)}(\liezx)$, by
 proposition \ref{prop.R2.orbits}.  In fact, there is a differentiable
 lift $\hat{\alpha}$ of $\alpha$ satisfying these conditions.  Note
 that by proposition \ref{prop.unipotent.extension} (1), the image of $\hat{\alpha}$ is contained in $\overline{\mathcal{R}}$.

Because $\lieu_+ \subset \omega_{\hat{\alpha}(t)}(\liezx)$ for all
$t$, the latter subspace, which is 2-dimensional, is abelian, a
standard fact that can be deduced from part (2) of the second axiom
for $\omega$.
Together with a generator $U$ of $\lieu_+$, it is spanned by an element of the form $E_1 + A + \xi$ with $A \in \liea$
and $\xi \in \lieq'_1$, the centralizer of $\lieu_+$ in $\liep^+$, which is contained in $\lieq_1$.
Note that centralizing $\lieu_+$ means $A \in \ker \beta$.  An argument similar to that in the proof of 
lemma \ref{lem:1d.degenerate} shows that if $\alpha(A) \neq 0$, then an element of $\liezx$ would be 
timelike somewhere (see also \cite[Lem 6.13]{mp.confdambra}), contradiciting proposition \ref{prop.R2.orbits}. 
Therefore $A = 0$, and $\omega_{\hat{\alpha}(t)}(\liezx)$ is spanned by $U$ and $E_1 + \xi$ for
some $\xi \in \lieq'_1$.

On the other hand, if $A_{t_0} = \omega(\hat{\alpha}'(t_0))$ were transverse to $E_1^\perp + \lieq$, 
then calculations with the formula
$$ \left. \frac{d}{dt} \right|_{t_0} \omega_{\hat{\alpha}(t)}(Z) = [A_{t_0}, \omega_{\hat{\alpha}(t_0)}(Z)] \qquad Z \in \liezx$$
(see the end of \cite[Sec. 6.4]{mp.confdambra}) 
would give that $\omega_{\hat{\alpha}(t)}(\liezx)$ is not contained 
in the subspace spanned by $U$ and $E_1 + \xi$ for $t$ close to $t_0$.  
From this contradiction, we conclude that $\hat{\alpha}'(t) \in \hat{\mathcal{D}}$ 
for all $t$.  It follows that $\hat{\alpha}$ is contained in an integral leaf of $\hat{\mathcal{D}}$, 
and so $\alpha$ is contained in $\mathcal{P}_y$ for $t$ near $t_0$.  Varying $\alpha$ over 
paths through $y$ in $\Sigma^{(2)}$ gives that a neighborhood of $y$ in $\Sigma^{(2)}$ is contained in a plaque $\mathcal{P}_y$. 
Because both sets are smooth surfaces near $y$, we can shrink $\mathcal{P}_y$ so that it is contained in $\Sigma^{(2)}$.
\end{preuve}

Proposition \ref{prop.no-accumulation} together with proposition
\ref{prop.accumulation} shows that no $1$-dimensional orbit of 
$\Sigma^{(2)}$ can be accumulated by a $2$-dimensional orbit, because
the latter is contained in $\Omega$. By lemma \ref{lem.stratification}, the  complement of $\Sigma^{(2)}$
in $\Sigma$ is a finite union of  $1$-dimensional orbits. We infer that only finitely many $1$-dimensional orbits can be accumulated by $2$-dimensional 
orbits.  This is a contradiction to corollary \ref{coro.uncountable}.

\section{Case $\liezx$ is $1$-dimensional}
\label{sec.1d.proof}

The last case, in which $\liezx = \RR X$, is shown in this
section to also lead to conformal flatness.  The 1-dimensional orbits
of $Z_X = \{ \varphi^t_X \}$ cannot be closed by theorem
\ref{thm.isotropy.algebraic}, so there must be fixed points by
theorem \ref{thm.gromov.stratification}.   
This last part of the proof is very similar to section 5, treating global fixed points,
in \cite{mp.confdambra}, with slight simplification here stemming from
$Z_X$ being 1-dimensional.

We first recall the model space, the $3$-dimensional Einstein universe $\Ein^{1,2}$, which 
 was introduced at the beginning of section \ref{sec.symmetries}.
 This space can be realized as the projectivization of 
  the lightcone for the quadratic form 
  $$ Q^{2,3}(x)=2x_0x_4+2x_1x_3+x_2^2$$
  on $\RR^5$ with basis $E_0, \ldots, E_4$.
It is a compact projective variety and inherits a Lorentzian conformal
structure from $Q^{2,3}$.  We will denote $o=[E_0] \in \Ein^{1,2}$
and take $P$ to be the stabilizer of $o$ in $\PO(2,3)$.

The following proposition will conclude the proof of theorem \ref{thm.main.theorem}.

\begin{proposition} Suppose that $\liez_X = \RR X$.  Then $(M,g)$ is conformally flat.
\end{proposition}

The proof occupies the remainder of this section.

\subsection{Unipotent or hyperbolic isotropy}

Let $x_0$ be a $Z_X$-fixed point. 
By theorem \ref{thm.fm.linearization},
$M$ is conformally flat or $X$ is linearizable at $x_0$.  Assuming
the latter, let $\hat{x}_0 \in \pi^{-1}(x_0)$ with
$\omega_{\hat{x}_0}(X) = \hat{X} \in \lieg_0$.
Recall from theorem \ref{thm.isotropy.algebraic} that
$(\hat{I}_X)_{\hat{x}}$ is algebraic, which means that it is closed under Jordan
decomposition (see, eg, \cite[Thm 4.3.3]{morris.ratners.thms.book}).  Thus $\hat{X}$ is
$\RR$-semisimple, nilpotent, or elliptic.
The elliptic case would contradict
noncompactness of $Z_X$, since the isotropy monomorphism $\iota_{\hat{x}} :
\mbox{Is}^{loc}(x) \rightarrow P$ (see section \ref{sec.gromov.frobenius}) is proper.

\subsection{Hyperbolic isotropy implies conformal flatness}

Suppose $\hat{X}$ is semisimple, so it generates an $\RR$-split
1-parameter subgroup of $G_0$.  If this subgroup were balanced or
contracting, then $(M,g)$ would be conformally flat by proposition \ref{prop.stable.flat}.
Up to conjugation in $G_0$ or replacing $\hat{X}$ with $-\hat{X}$, we may assume
$$ e^{t \hat{X}} : (x,y,z) \mapsto (e^{t(b-a)}x, 
e^{-a}y, e^{t(-a-b)}z) \qquad \mbox{with } 0 \leq a < b$$

Recall the exponential map of $\omega$ from subsection \ref{subsec.stability.propagation}. 
Let  $\hat{\alpha} = \exp_{\hat{x}_0} (s E_1)$, and let $\alpha = \pi
  \circ \hat{\alpha}$.  The set $\alpha^+ = \{\alpha(s) : s > 0 \}$ is
  a $Z_X$-orbit.  Let $x_1$ be an accumulation point of
  $\alpha(s_k)$ for a sequence $s_k \rightarrow \infty$.  It is shown
  in Section 5.2 of \cite{mp.confdambra}
  that the isotropy of $X$ at $x_1$ is also hyperbolic, with $a$ and
  $b$ exchanged.  We recall the main points of that argument here.

  \begin{lemme}[\cite{mp.confdambra} Lemma 5.4]
    $$\overline{\alpha}^+ = \{ x_0 \} \cup \alpha^+ \cup \{ x_1 \}$$
    and $x_1$ is a $Z_X$-fixed point.
  \end{lemme}

The proof is based on the reasonable topological properties of the
orbit closure $\overline{\alpha}^+$ guaranteed by semi-analyticity as
in theorem \ref{thm.gromov.stratification}.
A corollary is that $\overline{\alpha^+}$ admits a smooth
parametrization by a finite interval.

To compute the isotropy of $\{ \varphi^t_X \}$ at $x_1$, we compare
with the flow $\{ e^{t \hat{X}} \}$ along the lightlike
geodesic of $\mbox{Ein}^{1,2}$ corresponding to $\alpha$.
The path
$\hat{\alpha}$ has a \emph{development} in $\mbox{PO}(2,3)$, which is
the 1-parameter subgroup generated by $E_1$, viewed as an element of
$\lieg_{-1} \subset \lieg$.  The projection of this path to
$\mbox{PO}(2,3)/P \cong \mbox{Ein}^{1,2}$
is the curve $e^{s E_1}.o$.  In the homogeneous
coordinates on ${\bf RP}^4$, this path is $[1:s:0:0:0]$; it lies in the projectivization of
the totally isotropic plane $P_0 = \mbox{span} \{ e_0,e_1 \} \subset \RR^{2,3}$.

Take $R \in \lieg$ to be the generator of a rotation in the
plane $P_0 $.  Note that $\{  e^{\theta R} .o \}$ coincides with
$\alpha$ for $0 \leq \theta < \pi/2$.
Define
$\hat{\beta}(\theta) = \exp_{\hat{x}_0}(\theta R)$ and $\beta = \pi
\circ \hat{\beta}$.  There is a
path $\ell_\theta$ in $P$ such that
$$ \hat{\beta}(\theta) = \hat{\alpha}(s).\ell_\theta \qquad s= \tan
\theta, \ 0 \leq \theta < \pi/2$$

\begin{lemme}[\cite{mp.confdambra} Lemma 5.6] There is $\hat{x}_1 \in
  \pi^{-1}(x_1)$ equal to $\lim_{\theta \rightarrow \pi/2}
  \hat{\beta}(\theta)$.
  \end{lemme}

  There is a smooth lift $\tilde{\beta}$ of $\beta$ to $\hat{M}$,
  related to $\hat{\beta}$ by
$$ \tilde{\beta}(\theta) = \hat{\beta}(\theta).p_\theta \qquad 0 \leq
\theta < \pi/2$$
for a path $p_\theta$ in $P$ defined on $[0,\pi/2)$.
It is shown that $\omega_P(p_\theta')$ satisfies a linear, first-order,
inhomogeneous ODE with coefficients all in terms of $\tilde{\beta}$
and its derivatives.
It follows that the
$\omega_P$-derivative of $p_\theta$ is bounded on $[0,\pi/2)$.  As $P$
is complete with respect to the left-invariant framing determined by
$\omega_P$,
 the limit of $p_\theta$ as $\theta \rightarrow
\pi/2$ exists, thus defining the corresponding limit of
$\hat{\beta}(\theta)$.

\begin{lemme}[\cite{mp.confdambra} Lemma 5.7]
The isotropy of $\varphi^t_X $  with respect to $\hat{x}_1$ is
$$ g^t = \left(
  \begin{array}{ccccc}
    e^{tb} &   &   &  &  \\
           &  e^{ta} &  &   & \\
           &   &  1  &  &   \\
           &   &   &   e^{-ta}  &  \\
           &   &  &   &  e^{-tb}
  \end{array}
\right) \in G_0 < \mbox{PO}(2,3)
$$
\end{lemme}

This value is obtained from the limit as $\theta
\rightarrow \pi/2$ of the holonomy of $\{ \varphi^t_X
\}$ at $\hat{\beta}(\theta)$.
A shortcut to computing this limit is given by computing in $\mbox{PO}(2,3)$ the conjugate $e^{- (\pi/2)\cdot R} e^{t \hat{X}} e^{(\pi/2)  \cdot R} = g^t$.

Now the isotropy $(\hat{I}_X)_{\hat{x}_1} = \{ g^t \}$ is contracting,
so $(M,g)$ is conformally flat by proposition \ref{prop.stable.flat}.

\subsection{All unipotent isotropy not possible}

Now suppose that the isotropy of $X$ at $\hat{x}_0$ is a nilpotent
element $\hat{X} \in \lieg_0$ with
respect to $\hat{x}_0 \in \pi^{-1}(x_0)$.  Here we briefly recall the proof of \cite[Prop 5.8]{mp.confdambra}.

Take $\mathcal{U} \subset \lieg_{-1}$ such that $\pi \circ
\exp_{\hat{x}_0}$ is a diffeomorphism of $\mathcal{U}$ onto its
image, which we denote $U$.  For $v \neq 0$ in $\mathcal{U}$, let $y_0
= \exp_{\hat{x}_0}(v)$.  The connected component
of $y_0$ in $Z_X.y_0 \cap U$ equals $\pi \circ
\exp_{\hat{x}_0}((\hat{I}_X)_{\hat{x}_0}.v \cap \mathcal{U})$ and is bounded away from
$x_0$.  By theorem \ref{thm.gromov.stratification}, $\overline{Z_X.y_0}$ is
locally connected and contains $Z_X.y_0$ as a relatively open subset.
Therefore, $x_0 \notin \overline{Z_X.y_0}$ for any $y_0 \neq x_0$ in
$U$.

Choose $v$ not fixed by the isotropy.  For $y_0$ as above, $\overline{Z_X.y_0}$ contains
a point $x_1 \neq y_0$ with closed $Z_X$-orbit by theorem \ref{thm.gromov.stratification}. This closed orbit is necessarily 
a fixed
point by theorem \ref{thm.isotropy.algebraic}.  We can assume, by the results above for hyperbolic isotropy,
that the isotropy of $Z_X$ is again unipotent at
$x_1$.   Now for $y_1 \in Z_X.y_0$ sufficiently close to $x_1$, we have
$x_1 \in \overline{Z_X.y_1}$.  On the other hand, the argument in the
previous paragraph gives that $Z_X.y_1$ is bounded away from $x_1$---contradiction.

\section{Appendix: Proof of theorem \ref{thm.translation.holonomy}}
\label{sec.appendix}
This appendix is devoted to the proof of theorem
\ref{thm.translation.holonomy}. We will actually prove a more general
statement, for smooth manifolds and sequences of local conformal transformations.

 \begin{theoreme}
     \label{thm.translation.holonomy.smooth}
Let $(M,g)$ be a smooth, compact, 3-dimensional Lorentzian manifold.  Let $\{
f_k \} \subset \mbox{Conf}^{loc}(M,[g])$ be an unbounded sequence
defined on a common neighborhood $U$ of $x \in M$.  If $\{ f_k \}$
admits a holonomy sequence at $x$ contained in $P^+$, then
 there exists a nonempty open subset $U \subset M$ which  is conformally flat.
   \end{theoreme}

 It is clear that theorem \ref{thm.translation.holonomy} follows directly from theorem \ref{thm.translation.holonomy.smooth}
\subsection{Conformal geodesic segments}
\label{sec.conformal.geodesic}
 
The strategy to prove theorem \ref{thm.translation.holonomy.smooth} is to exhibit dynamical properties of the sequence 
$\{ f_k \}$
 which force conformal flatness on an open subset. The dynamical
 behavior of $\{ f_k \}$ around a point $x$ is understood via
 the action of its holonomy sequences on \emph{conformal geodesics} in
 the model space, introduced below.

The $3$-dimensional Minkowski space will be taken to have the
quadratic form $q(x):=2x_1x_3+x_2^2$, and will be denoted by $\RR^{1,2}$.  
A conformal immersion $j^o: \RR^{1,2} \to \Ein^{1,2}$ is given in
homogeneous coordinates on ${\bf RP}^4$ by the formula
$$ j^o : x \mapsto \left[ 1:x_1:x_2:x_3:-\frac{q(x)}{2} \right],$$
mapping the origin in $\RR^{1,2}$ to $o \in \Ein^{1,2}$.
 For us a {\it conformal geodesic segment} of $\mbox{Ein}^{1,2}$
emanating from $o$ will be a curve $\alpha: [0,1] \to \Ein^{1,2}$
of the form 
  $$ s \mapsto p.j^o(sw),$$
  where $p \in P$ and $w \in \RR^{1,2}$.
  
 \subsection{Local dynamics via conformal geodesic segments}
  
 Let $\{ p_k\}$ be a holonomy sequence for $\{f_k\} \subset
 \mbox{Conf}^{loc}(M)$ at $x \in M$.
   It is an unbounded sequence of conformal transformations of $\Ein^{1,2}$
   fixing $o$, which in turn admits holonomy sequences at other points of $\Ein^{1,2}$. 
 The following proposition, borrowed from \cite{frances.locdyn}, explains how conformal geodesic segments relate holonomy sequences of $\{f_k\}$
  and $\{p_k\}$.

\begin{proposition}[see \cite{frances.locdyn} Prop 6.3]
 \label{prop.pullback}
 Let $(M,g)$ be a smooth Lorentzian manifold. Let $\{ f_k \} <
 \mbox{Conf}(M,[g])$ with holonomy sequence $\{ p_k \}$
along $\hx_k \to \hx$ in $\hat{M}$. Assume that there exists a conformal geodesic segment
  $\beta: [0,1] \to \Ein^{1,n-1}$ emanating from $o$ such that
   $\lim_{k \to  \infty}p_k.[\beta]=o$.
    Then any pointwise holonomy  sequence of $\{ p_k \}$ at $\beta(1)$ admits a subsequence which is a holonomy sequence 
  for  $\{ f_k \}$ with respect to some converging sequence $\hy_k \to \hy$ in $\hat{M}$.
\end{proposition}

This proposition, together with  proposition \ref{prop.stable.flat},
brings us closer to theorem \ref{thm.translation.holonomy.smooth}, through the following corollary:
\begin{corollaire}
 \label{coro.flat}
 Let $(M,g)$ be a smooth, $3$-dimensional, Lorentzian manifold. Let $\{ p_k \}$
  be a holonomy sequence for $\{ f_k \}$ along $\hx_k \to \hx$ in $\hat{M}$. Suppose there exists a conformal geodesic segment
  $\beta: [0,1] \to \Ein^{1,2}$ emanating from $o$ such that
   $\lim_{k \to  \infty}p_k.[\beta]=o$. If $\{ p_k \}$ admits a pointwise holonomy sequence at $\beta(1)$
    which is stable, then a nonempty open subset $U \subset M$ is conformally flat. 
\end{corollaire}

Recall definition \ref{def:holonomy_trichotomy} for stable holonomy sequences.
 
 \subsection{Lemma ensuring stable holonomy sequence}
 \label{sec.technical.lemma}
 
 Theorem \ref{thm.translation.holonomy.smooth} is a direct consequence of Corollary \ref{coro.flat} and 
 Lemma \ref{lem.translations} below, which is  the main technical result of this section.

\begin{lemme}
 \label{lem.translations}
 Let $\{ p_k \}$ be a sequence of $P^+$. After passing to a subsequence, there exists 
  a conformal geodesic segment $\beta: [0,1] \to \Ein^{1,2}$ emanating
 from $o$ such that  $\lim_{k \to  \infty}p_k.[\beta]=o$, and such that $\{ p_k\}$
  admits a pointwise holonomy sequence at $\beta(1)$ which is stable. 
\end{lemme}

\begin{preuve}
  Denote the Euclidean norm on $\RR^3$ by $\lVert \cdot \rVert$.
  Write
  $$p_k=\left(  \begin{array}{ccc}
1 & t_kv_k^* & -\frac{t_k^2 q(v_k)}{2}\\
0 & I_3 & -t_kv_k \\
0&0&1\\
                       
                     \end{array}
 \right ) \in P^+ < \mbox{O}(2,3).$$
Here $v_k$ is a sequence of $\RR^{1,2}$ satisfying $\lVert v_k \rVert =1$, and $t_k \geq 0$.
 The expression $v_k^*$ stands for $v_k^t {\mathbb I}$, where
 ${\mathbb I}$ is as in section \ref{sec.cotton.tensor}.
 Observe that $\{ p_k \}$, hence $\{ t_k \}$, is unbounded, because
 $\{ f_k \}$ is unbounded. After taking a subsequence, we assume that $t_k \to  \infty$, and 
  that there is a vector $v = \lim v_k$.
 
Recall the  conformal immersion $j^o: \RR^{1,2} \to \Ein^{1,2}$ from
section \ref{sec.conformal.geodesic} above.
 Let $x_k \rightarrow x \in \RR^{1,2}$. For $u \in [0,1]$, and $k \in \NN$,
   define $\beta_k(u):=j^o(ux_k)$, and $\beta(u):=j^o(ux)$. Observe that each $\beta_k$ and $\beta$ are conformal geodesic segments 
  emanating from $o$. 
 Define $x_k(u):=(j^o)^{-1}(p_k.\beta_k(u))$.  From the matrix expression of $p_k$, and the formula for $j^o$, 

\begin{equation}
\label{eq.formule1}
x_k(u) = \left( 1 + t_ku\langle x_k,v_k \rangle + \frac{ t_k^2 u^2
    q(v_k)q(x_k)}{4} \right)^{-1} \cdot \left(ux_k + \frac{t_k u^2q(x_k)}{2} v_k \right).
\end{equation}

After possibly taking a further subsequence of $\{ p_k \}$, we may assume that $ t_k |q(v_k)|$ converges in
$[0,  \infty]$.  There are two subcases:

\subsubsection*{First case: $t_k |q(v_k)| \rightarrow  \infty$.}

In this case, $q(v_k)$ is nonzero for  $k$ large
enough, so that after perhaps taking a subsequence of $\{ p_k \}$,
 we may assume that the sign of $q(v_k)$ is constant.  Choose
 $\epsilon=\pm 1$ so that $\epsilon q(v_k) \geq 0$ for all
 sufficiently large $k$.
 Choose $x$ such that 
 \begin{enumerate}
  \item[(a)] $\epsilon q(x)$ is positive.
  \item[(b)] $\langle x,v \rangle$ is positive.
  \end{enumerate}

%
%

Let $u \in (0,1]$ and write:
 $$1 + t_ku\langle x_k,v_k \rangle + \frac{ t_k^2 u^2 q(v_k)q(x_k)}{4}=u^2 t_k^2 \epsilon q(v_k) \left( \frac{\epsilon q(x_k)}{4}+
 \frac{\langle x_k,v_k \rangle}{u t_k  \epsilon q(v_k)} +\frac{1}{ u^2t_k^2 \epsilon q(v_k)}\right).$$
 Under the assumption $t_k |q(v_k)| \to  \infty$,
 $$\lim_{k \to  \infty} \left( \frac{\epsilon q(x_k)}{4}+
 \frac{\langle x_k,v_k \rangle}{u  t_k \epsilon q(v_k)} +\frac{1}{ u^2t_k^2 \epsilon q(v_k)}\right)=\frac{\epsilon q(x)}{4},$$
 so that
  for $k$ big enough, 
  $$\left( \frac{\epsilon q(x_k)}{4}+
 \frac{\langle x_k,v_k \rangle}{u  t_k \epsilon q(v_k)} +\frac{1}{ u^2t_k^2 \epsilon q(v_k)}\right) \geq \frac{\epsilon q(x_k)}{8}.$$
 
 If $t_ku \geq 1$, we can infer from (\ref{eq.formule1}) and the previous inequalities that
 \begin{equation}
 \label{eq.formule3}
  \lVert x_k(u) \rVert \leq \frac{8}{ t_k q(v_k)q(x_k)} \left( \lVert
    x_k \rVert + \frac{|q(x_k)|}{2}   \right).
 \end{equation}
Observe that (\ref{eq.formule3}) also holds trivially if $u=0$.
 
Suppose $t_ku \leq 1$, and note that conditions $(a)$ and $(b)$ on $x$ imply that for  $k$ big enough,
$$1 + t_ku\langle x_k,v_k \rangle + \frac{ t_k^2 u^2 q(v_k)q(x_k)}{4} \geq 1.$$
 Then for $k$ large enough,
 \begin{equation}
 \label{eq.formule4}
 \lVert x_k(u) \rVert \leq u \lVert x_k \rVert + \frac{u^2 t_k
   |q(x_k)| }{2} \leq \frac{1}{t_k} \left( \Vert x_k \rVert + \frac{|q(x_k)|}{2}  \right)
\end{equation}
From 
(\ref{eq.formule4}), we infer:
 \begin{equation}
  \label{eq.formule5}
  \lim_{k \to  \infty}{\sup_{u \in [0,1]} \lVert x_k(u) \rVert }=0.
 \end{equation}

Taking  $x_k \equiv x$ gives $\lim_{k \to  \infty}p_k([\beta]) = o.$
Moreover, (\ref{eq.formule5}) shows that
$$\lim_{k \to  \infty}p_k(j^o(x_k))=o$$
for any $x_k\rightarrow x$.
In particular, $\{ p_k \}$ is stable at $j^o(x)=\beta(1)$ in the sense
of \cite[Def. 4.1]{frances.degenerescence}.
 It was proved in \cite{frances.degenerescence} Lemma $4.3$, that if
 $\{ p_k \}$ has this stability property at some point $z$,  it admits a pointwise 
 stable holonomy sequence at $z$. Lemma \ref{lem.translations} is thus
 proved in this case. 


\subsubsection*{Second case: $\lim_{k \to  \infty} t_kq(v_k)=a \in \RR$}

This time, take $x \in \RR^{1,2}$
 such that:
 \begin{enumerate}
  \item[(a)] $q(x)=0$.
  \item[(b)]  $\langle x,v \rangle > 0$ 
 \end{enumerate}

 Let $u \in (0,1]$, and write 
$$1 + t_ku\langle x_k,v_k \rangle + \frac{ t_k^2 u^2
  q(v_k)q(x_k)}{4}=1 +t_k u \left(  \langle x_k,v_k \rangle + \frac{u}{4}t_kq(v_k)q(x_k) \right).$$
Because $ \langle x_k,v_k \rangle + \frac{u}{4}t_kq(v_k)q(x_k) $
 tends to $\langle x,v \rangle$ as $k \to \infty$,
 $$1 + t_ku\langle x_k,v_k \rangle + \frac{ t_k^2 u^2 q(v_k)q(x_k)}{4} 
 \geq \frac{1}{2}t_ku \langle x,v \rangle \qquad \mbox{for } k \
 \mbox{large enough}.$$ 
  
  We infer from (\ref{eq.formule1}) that:
  \begin{equation}
   \label{eq.formule6}
   \Vert x_k(u) \rVert \leq \frac{2}{\langle x,v \rangle } \left(
     \frac{\lVert  x_k \rVert }{t_k} + \frac{|q(x_k)|}{2}\right).
  \end{equation}

  This inequality holds trivially for $u=0$, giving
  $$\lim_{k \to  \infty}{\sup_{u \in [0,1]} \lVert x_k(u) \rVert}=0.$$
  Taking $x_k \equiv x$ gives
  $\lim_{k \to  \infty}p_k([\beta])=o$.
Moreover, if  $\{ x_k \}$ is any sequence converging to $x$, inequality (\ref{eq.formule6})  shows that
$\lim_{k \to \infty}p_k(j^o(x_k))=o $. As in the first case, this implies that $\{ p_k \}$ admits a pointwise, 
stable holonomy sequence at $j^o(x)=\beta(1)$, and Lemma \ref{lem.translations} is proved. 
\end{preuve}

\bigskip

\bibliographystyle{amsplain}
\bibliography{karinsrefs}

\providecommand{\bysame}{\leavevmode\hbox to3em{\hrulefill}\thinspace}
\providecommand{\MR}{\relax\ifhmode\unskip\space\fi MR }
\providecommand{\MRhref}[2]{%
  \href{http://www.ams.org/mathscinet-getitem?mr=#1}{#2}
}
\providecommand{\href}[2]{#2}
\begin{thebibliography}{10}

\bibitem{amores.killing}
A.~M. Amores, \emph{Vector fields of a finite type ${G}$-structure}, J. Diff.
  Geom. \textbf{14} (1979), no.~1, 1--16.

\bibitem{bn.simpleconf}
U.~Bader and A.~Nevo, \emph{Conformal actions of simple {L}ie groups on compact
  pseudo-{R}iemannian manifolds}, J. Differential Geom. \textbf{60} (2002),
  no.~3, 355--387.

\bibitem{bierstone.milman}
E.~Bierstone and P.D. Milman, \emph{Semianalytic and subanalytic sets}, IHES
  Publ. Math. \textbf{67} (1988), 5--42.

\bibitem{calvaruso.kowalski.ricci.oper}
G.~Calvaruso and O.~Kowalski, \emph{On the {R}icci operator of locally
  homogeneous {L}orentzian 3-manifolds}, Cent. Eur. J. Math. \textbf{7} (2009),
  no.~1, 124--139.

\bibitem{cap.me.parabolictrans}
A.~{\v C}ap and K.~Melnick, \emph{Essential {K}illing fields of parabolic
  geometries}, Indiana Univ. Math. J. \textbf{62} (2013), no.~6, 1917--1953.

\bibitem{cap.slovak.book.vol1}
A.~{\v{C}}ap and J.~Slov\'ak, \emph{Parabolic geometries {I}}, Mathematical
  Surveys and Monographs, vol. 154, American Mathematical Society, Providence,
  RI, 2009.

\bibitem{coley.hervik.pelavas.3d}
A.~Coley, S.~Hervik, and N.~Pelavas, \emph{Lorentzian spacetimes with constant
  curvature invariants in three dimensions}, Class. Quantum Grav. \textbf{25}
  (2008), no.~2.

\bibitem{dambra.lorisom}
G.~D'Ambra, \emph{Isometry groups of {L}orentz manifolds}, Invent. Math.
  \textbf{92} (1988), no.~3, 555--565.

\bibitem{dag.rgs}
G.~D'Ambra and M.~Gromov, \emph{Lectures on transformation groups: geometry and
  dynamics}, Surveys in Differential Geometry (Lehigh University, Bethlehem,
  PA, 1990) (Cambrige, MA), 1991, pp.~19--111.

\bibitem{dz.lor3d.lochom}
S.~Dumitrescu and A.~Zeghib, \emph{G\'eom\'etries lorentziennes de dimension 3:
  classification et compl\'etude}, Geom. Dedicata \textbf{149} (2010),
  243--273.

\bibitem{eisenhart.riem.geom}
L.P. Eisenhart, \emph{Riemannian geometry}, 2nd ed., Princeton Univ. Press,
  1949.

\bibitem{ferrand.histoire}
J.~Ferrand, \emph{Histoire de la r\'eductibilit\'e du groupe conforme des
  vari\'et\'es riemanniennes (1964-1994)}, S\'eminaire de th\'eorie spectrale
  et g\'eom\'etrie, vol.~17, Institut Fourier, 1998-99, pp.~9--25.

\bibitem{frances.nogloballich}
C.~Frances, \emph{Sur les vari\'et\'es lorentziennes dont le groupe conforme
  est essentiel}, Math. Ann. \textbf{332} (2005), no.~1, 103--119.

\bibitem{frances.ccvf}
\bysame, \emph{Causal conformal vector fields, and singularities of twistor
  spinors}, Ann. Global Anal. Geom. \textbf{32} (2007), no.~3, 277--295.

\bibitem{frances.lfrank1}
\bysame, \emph{Sur le groupe d'automorphismes des g\'eom\'etries paraboliques
  de rang 1}, Ann. Sci. \'Ecole Norm. Sup. (4) \textbf{40} (2007), no.~5,
  741--764.

\bibitem{frances.degenerescence}
\bysame, \emph{D\'eg\'enerescence locale des transformations conformes
  pseudo-riemanniennes}, Ann. Inst. Fourier \textbf{62} (2012), no.~5,
  1627--1669.

\bibitem{frances.locdyn}
\bysame, \emph{Local dynamics of conformal vector fields}, Geom. Ded.
  \textbf{158} (2012), 35--59.

\bibitem{frances.pq.counterexs}
\bysame, \emph{About pseudo-{R}iemannian {L}ichnerowicz conjecture}, Transform.
  Groups \textbf{20} (2015), no.~4, 1015--1022.

\bibitem{frances.lorentz.3d}
\bysame, \emph{Lorentz dynamics on closed 3-manifolds},
  arxiv.org/abs/1804.08695, 2018.

\bibitem{frances.open.dense}
\bysame, \emph{Variations on {G}romov's open-dense orbit theorem}, Bull. Soc.
  Math. France \textbf{146} (2018), no.~4, 713--744.

\bibitem{fm.nilpconf}
C.~Frances and K.~Melnick, \emph{Nilpotent groups of conformal flows on compact
  pseudo-{R}iemannian manifolds}, Duke Math. J. \textbf{153} (2010), no.~3,
  511--550.

\bibitem{fm.champsconfs}
\bysame, \emph{Formes normales pour les champs conformes pseudo-riemanniens},
  Bull. Soc. Math. France \textbf{141} (2013), no.~3, 377--421.

\bibitem{fz.simpleconf}
C.~Frances and A.~Zeghib, \emph{Some remarks on conformal pseudo-{R}iemannian
  actions of simple {L}ie groups}, Math. Res. Lett. \textbf{12} (2005), 49--56.

\bibitem{gromov.rgs}
M.~Gromov, \emph{Rigid transformations groups}, G\'eom\'etrie Diff\'erentielle
  (Paris, 1986) (D.~Bernard and Y.~Choquet-Bruhat, eds.), Hermann, Paris, 1988,
  pp.~65--139.

\bibitem{kobayashi.transf}
S.~Kobayashi, \emph{Transformation groups in differential geometry}, Springer,
  Berlin, 1995.

\bibitem{lf.lich}
J.~Lelong-Ferrand, \emph{Transformations conformes et quasiconformes des
  vari\'et\'es riemanniennes; application \`a la d\'emonstration d'une
  conjecture de {A}. {L}ichnerowicz}, Acad. Roy. Belg. Cl. Sci. M\'em.
  \textbf{39} (1971), no.~5, 5--44.

\bibitem{lf.noncompact}
\bysame, \emph{The action of conformal transformations on a {R}iemannian
  manifold}, Math. Ann. \textbf{304} (1996), 277--291.

\bibitem{lich.cr}
A.~Lichnerowicz, \emph{Sur les transformations conformes d'une vari\'et\'e
  riemannieinne compacte}, C.R. Acad. Sci. Paris \textbf{259} (1964), 697--700.

\bibitem{lojasiewicz.semianalytique}
S.~Lojasiewicz, \emph{Ensembles semi-analytiques}, IHES Lecture notes, 1965,
  \url{https://perso.univ-rennes1.fr/michel.coste/Lojasiewicz.pdf}.

\bibitem{me.frobenius}
K.~Melnick, \emph{A {F}robenius theorem for {C}artan geometries, with
  applications}, Enseign. Math. S\'er. II \textbf{57} (2011), no.~1-2, 57--89.

\bibitem{mp.confdambra}
K.~Melnick and V.~Pecastaing, \emph{The conformal group of a compact simply
  connected {L}orentzian manifold}, arxiv.org/abs/1911.06251, 2019.

\bibitem{morris.ratners.thms.book}
D.~Witte Morris, \emph{Ratner's theorems on unipotent flows}, University of
  Chicago Press, Chicago, 2005.

\bibitem{obata.lich}
M.~Obata, \emph{The conjectures on conformal transformations of {R}iemannian
  manifolds}, J. Differential Geom. \textbf{6} (1971), 247--258.

\bibitem{pecastaing.frobenius}
V.~Pecastaing, \emph{On two theorems about local automorphisms of geometric
  structures}, Ann. Inst. Fourier (Grenoble) \textbf{66} (2016), no.~1,
  175--208.

\bibitem{pecastaing.smooth.sl2r}
\bysame, \emph{Lorentzian manifolds with a conformal action of {S}{L}(2,{{\bf
  R}})}, Comentarii Mathematici Helvetici \textbf{93} (2018), no.~2, 401--439.

\bibitem{rahmani.lorentz.heis}
N.~Rahmani and S.~Rahmani, \emph{Lorentz geometry of the {H}eisenberg group},
  Geom.Dedicata \textbf{118} (2006), 133--140.

\bibitem{schoen.cr}
R.~Schoen, \emph{On the conformal and {CR} automorphism groups}, Geom. Funct.
  Anal. \textbf{5} (1995), no.~2, 464--481.

\bibitem{sekigawa.curv.homog.3d}
K.~Sekigawa, \emph{On some 3-dimensional curvature homogeneous spaces}, Tensor
  N.S. \textbf{31} (1977), 87--97.

\bibitem{sharpe}
R.W. Sharpe, \emph{Differential geometry : {C}artan's generalization of
  {K}lein's {E}rlangen program}, Springer, New York, 1996.

\bibitem{zeghib.lorentz.3d}
A.~Zeghib, \emph{Killing fields in compact {L}orentz 3-manifolds}, J.
  Differential Geom. \textbf{43} (1996), 859--894.

\bibitem{zeghib.tgl1}
\bysame, \emph{Isometry groups and geodesic foliations of {L}orentz manifolds
  {I}: Foundations of {L}orentz dynamics}, Geom. Funct. Anal. \textbf{9}
  (1999), no.~4, 775--822.

\bibitem{zimmer.rankbounds}
R.J. Zimmer, \emph{Split rank and semisimple automorphism groups of
  ${G}$-structures}, J. Differential Geom. \textbf{26} (1987), no.~1, 169--173.

\end{thebibliography}

\begin{tabular}{lll}
Karin Melnick  & \quad\qquad & Charles Frances \\
Department of Mathematics & \quad\qquad & Institut de Recherche Math\'ematique Avanc\'ee  \\
4176 Campus Drive & \qquad \qquad & 7 rue Ren\'e-Descartes \\
University of Maryland & \quad\qquad & Universit\'e de Strasbourg \\
College Park, MD 20742 &\quad \qquad & 67085 Strasbourg Cedex  \\
USA &\quad \qquad & France \\
karin@math.umd.edu &\quad \qquad & frances@math.unistra.fr
\end{tabular}

\end{document}